\documentclass[11pt]{amsart}

\usepackage[T1]{fontenc}
\usepackage{lmodern}
\usepackage{microtype}
\usepackage{mathtools,amssymb,amsthm}
\usepackage{enumitem}
\usepackage{float}
\usepackage{xcolor}
\usepackage{tikz}
\usepackage{pgfplots}
\usetikzlibrary{positioning,calc,arrows.meta}
\pgfplotsset{compat=1.18}
\usepackage{aliascnt}
\usepackage{hyperref}
\usepackage[nameinlink,capitalize,noabbrev]{cleveref}
\usepackage[margin=1in]{geometry}

\hypersetup{
 colorlinks=true,
 linkcolor=blue!55!black,
 citecolor=blue!55!black,
 urlcolor=blue!55!black,
 pdftitle={Poisson laws and exterior stability for random alternating tensors},
 pdfauthor={Pakin Methawisal}
}

\newtheorem{theorem}{Theorem}[section]

\newaliascnt{proposition}{theorem}
\newtheorem{proposition}[proposition]{Proposition}
\aliascntresetthe{proposition}

\newaliascnt{lemma}{theorem}
\newtheorem{lemma}[lemma]{Lemma}
\aliascntresetthe{lemma}

\newaliascnt{corollary}{theorem}
\newtheorem{corollary}[corollary]{Corollary}
\aliascntresetthe{corollary}
\newaliascnt{conjecture}{theorem}
\newtheorem{conjecture}[conjecture]{Conjecture}
\aliascntresetthe{conjecture}

\crefname{conjecture}{Conjecture}{Conjectures}

\theoremstyle{remark}
\newaliascnt{remark}{theorem}
\newtheorem{remark}[remark]{Remark}
\aliascntresetthe{remark}

\newcommand{\F}{\mathbb F}
\newcommand{\E}{\mathbb E}
\newcommand{\Pp}{\mathbb P}
\newcommand{\Gr}{\operatorname{Gr}}
\newcommand{\Span}{\operatorname{span}}
\newcommand{\Pois}{\operatorname{Poisson}}
\newcommand{\qbinom}[2]{\genfrac{[}{]}{0pt}{}{#1}{#2}_{q}}

\title[Poisson laws for random alternating tensors]{Poisson laws and exterior stability for random alternating tensors}
\author{Pakin Methawisal}
\address{Mahidol University International College, Mahidol University,
999 Phutthamonthon 4 Rd., Salaya, Phutthamonthon,
Nakhon Pathom 73170, Thailand}
\email{pakin.met@student.mahidol.edu}
\date{}

\begin{document}

\begin{abstract}
For fixed $k\ge3$, we determine the critical law of totally isotropic $r$-spaces for a uniform random map $\Theta_N:\Lambda^k\F_q^N\to\F_q^m$. At the exact balance $m\binom rk=r(N-r)$, the entire null configuration is asymptotically an independent Bernoulli subset of $\Gr(r,\F_q^N)$ in total variation, uniformly in $q$ and $m$. Consequently, the counting measure $\Xi_r$ is asymptotically a Poisson point process, and its total mass $X_{N,r}$ is asymptotically Poisson. An exterior-rank stability theorem shows that near-extremal families decompose into Grassmann clusters with uniformly controlled span deficiency. We obtain quantitative rates and identify the first exterior-dependence scale. We also show that rare null $(r+1)$-spaces force high-order factorial-moment divergence, while the fixed-target bilinear cases $m=1,2$ exhibit non-Poisson critical behavior.
\end{abstract}

\maketitle

\section{Introduction}
Let $V_N=\F_q^N$ and choose uniformly $\Theta_N\in\operatorname{Hom}(\Lambda^kV_N,\F_q^m)$. Here $\Lambda^kV_N$ is the $k$th exterior power of $V_N$; equivalently, $\Theta_N$ is an $m$-tuple of alternating $k$-linear forms. By an $r$-space we mean an $r$-dimensional linear subspace of $V_N$. Such an $r$-space $H\le V_N$ is \emph{null} (equivalently, totally isotropic for $\Theta_N$) when $\Theta_N(\Lambda^kH)=0$. Write $\mathcal Z_{N,r}:=\{H\in\Gr(r,V_N):\Theta_N(\Lambda^kH)=0\}$ for the random null configuration, put $X_{N,r}=|\mathcal Z_{N,r}|$, and let $\alpha_N$ be the maximum null dimension. Since a fixed $r$-space is null with probability $q^{-m\binom rk}$,
\begin{equation}
 \mu_{N,r}:=\E X_{N,r}=\qbinom Nr q^{-m\binom rk},\qquad
 \log_q\mu_{N,r}=r(N-r)-m\binom rk+O(1).
 \label{eq:first-moment}
\end{equation}
The first moment already identifies the transition. Below the critical dimension the expected number of null $r$-spaces tends to infinity, suggesting that such a space should exist with high probability; above it the expectation tends to zero, so none exists with high probability. At the critical balance the expected count remains of constant order, and existence becomes random. The main result of the paper is a random-set description of this critical point.

Writing $N=r+c$, the Grassmannian contributes $r(N-r)=rc$ dimensions, while nullity imposes $m\binom rk$ linear conditions. We call the exact balance $m\binom rk=rc$ \emph{resonance}. Under this resonance condition, $c = \frac mr\binom rk \in \mathbb Z$. At resonance we prove the stronger statement that the whole random null configuration is asymptotically an independent Bernoulli subset of the Grassmannian; the scalar Poisson law is its cardinality projection.

The main obstruction to an immediate Poisson argument is dependence: distinct $r$-spaces may overlap, so their null events need not be independent. The role of the exterior-rank and stability theory is to control precisely these dependent configurations. Roughly, a dependent family either incurs a large exterior-rank penalty, making simultaneous nullity very unlikely, or it has small excess and is forced into a Grassmann-cluster structure, making such configurations sufficiently sparse to count.

For a finite ordered family $\mathcal H=(H_1,\ldots,H_t)$ of distinct $r$-spaces, write
\[
 \mathcal O:=\sum_{H\in\mathcal H}H,
 \qquad
 \mathcal E:=\sum_{H\in\mathcal H}\Lambda^kH,
 \qquad
 \beta:=\frac1r\binom rk.
\]
Here $\mathcal O$ is the ordinary sum and $\mathcal E$ is the corresponding exterior sum. For $j\ge2$, write
\[
 \mathcal O_{<j}:=\sum_{i=1}^{j-1}H_i,
 \qquad
 \mathcal E_{<j}:=\sum_{i=1}^{j-1}\Lambda^kH_i,
 \quad
 d_j:=\dim(H_j\cap\mathcal O_{<j}),
 \quad
 e_j:=\dim(\Lambda^kH_j\cap\mathcal E_{<j}),
\]
so in particular $d_2=\dim(H_1\cap H_2)$. We also write $\mathfrak O:=\dim(\mathcal O)$ and $\mathfrak E:=\dim(\mathcal E)$. 

The probability that all $H_i$ are null is $q^{-m\mathfrak E}$, so the key comparison is between $\mathfrak E$ and $\mathfrak O$. Our exterior-rank inequality shows that every non-direct family satisfies $\mathfrak E\ge\beta\bigl(\mathfrak O+1\bigr)$. At resonance this gives an extra factor $q^{-c}$ in the probability that all $H_i$ are null. When $c/r^2\to\infty$, this extra factor already dominates the number of competing configurations; we treat that case first in \cref{sec:factorial-moments}. We call it the \emph{high-target} regime. The regime in which $m$ is fixed is the \emph{fixed-target} regime. Here, $c\asymp r^{k-1}$. For $k\ge4$ this again dominates the $r^2$ configuration scale. For $k=3$ the two scales coincide, so the trilinear case requires additional structure.

Define the exterior excess by
\[
 G_k(\mathcal H)=\mathfrak E-\beta \mathfrak O,
\]
and say that $\mathcal H$ is $\Gamma$-low-excess if $G_k(\mathcal H)\le\Gamma r^{k-1}$, and $\Gamma$-high-excess otherwise.

For $r$-spaces $H,A$, write $d_{\mathrm{Gr}}(H,A)=r-\dim(H\cap A)$ for the Grassmann graph distance.

The exterior-rank inequality proved in \cref{thm:exterior-intro} gives $\mathfrak E\ge\beta\mathfrak O$, with equality exactly when the ordinary sum is direct. Thus $G_k(\mathcal H)\ge0$ measures the excess above this extremal lower bound; \cref{thm:stability-intro} gives a uniform structural description of families for which this excess is small.

Unless stated otherwise, all asymptotic statements involving $r$ are understood to hold for sufficiently large $r$. Set $C_q:=\prod_{j\ge1}(1-q^{-j})^{-1}$, a reciprocal $q$-Pochhammer constant arising from the Gaussian-binomial asymptotics.

At resonance, define
\[
\mathcal G_r:=\Gr(r,V_N), \quad p_r:=q^{-m\binom{r}{k}}=q^{-rc}, \quad \rho_{m,k}(r):=\sqrt m\,r^{(k-2)/2}, \quad \Psi_r:=\rho_{m,k}(r)\log\!\bigl(2+\rho_{m,k}(r)\bigr).
\]

Let $\mathcal B_r\subseteq\mathcal G_r$ be the Bernoulli subset obtained by retaining each element independently with probability $p_r$, and let $\Pi_r$ be the Poisson point process on $\mathcal G_r$ with intensity $p_r\sum_{H\in\mathcal G_r}\delta_H$.

Identifying subsets of $\mathcal G_r$ with their simple counting measures, set $\Xi_r:=\sum_{H\in\mathcal Z_{N,r}}\delta_H$.

For a random object $Z$, let $\mathcal L(Z)$ denote its probabilistic law and $d_{\mathrm{TV}}$ total-variation distance. We use this notation throughout whenever the resonant parameters are understood.

\begin{theorem}[Uniform critical random-set law]
\label{thm:fixed-target-intro}
Fix $k\ge3$. There exist $a_k^{\mathrm{TV}}>0$, depending only on $k$, and an absolute constant $B>0$ such that, at resonance and uniformly over all prime powers $q$ and integers $m\ge1$,
\begin{align}
d_{\mathrm{TV}}\!\left(
\mathcal L(\mathcal Z_{N,r}),
\mathcal L(\mathcal B_r)
\right)
&\le e^{-a_k^{\mathrm{TV}}\Psi_r},
\label{eq:bernoulli-subset-tv}
\\
d_{\mathrm{TV}}\!\left(
\mathcal L(\Xi_r),
\mathcal L(\Pi_r)
\right)
&\le e^{-a_k^{\mathrm{TV}}\Psi_r},
\label{eq:poisson-process-tv}
\\
d_{\mathrm{TV}}\!\left(
\mathcal L(X_{N,r}),
\Pois(\mu_{N,r})
\right)
&\le e^{-a_k^{\mathrm{TV}}\Psi_r},
\label{eq:critical-count-poisson}
\\
d_{\mathrm{TV}}\!\left(
\mathcal L(X_{N,r}),
\Pois(C_q)
\right)
&\le Bq^{-r}+e^{-a_k^{\mathrm{TV}}\Psi_r}.
\label{eq:uniform-count-to-Cq}
\end{align}
The constants and the threshold in $r$ are uniform in $q$ and $m$. Consequently, along every resonant sequence $q=q_r$ and $m=m_r\ge1$,
\[
d_{\mathrm{TV}}\!\left(
\mathcal L(X_{N,r}),\Pois(C_{q_r})
\right)\longrightarrow0.
\]
If $q_r\to\infty$, then $C_{q_r}\to1$ and hence $X_{N,r}\xrightarrow d\Pois(1)$.
\end{theorem}
The proof is given in \cref{subsec:proof-uniform-critical-random-set}.

The scale $\rho_{m,k}(r):=\sqrt m\,r^{(k-2)/2}$ arises from the entropy--stability balance proved in \cref{sec:fixed-target}: growing exterior stability \cref{prop:growing-exterior-stability} yields factorial-moment control through order $t\asymp\rho_{m,k}(r)$; see \cref{prop:growing-fixed-target-moments}. Applying quantitative Brun inversion at factorial order $T\asymp\rho_{m,k}(r)$, followed by Stirling's formula for the resulting factorial remainder, gives $\exp\!\left(-\Omega_k\!\left(\Psi_r\right)\right)$; see \cref{subsec:proof-uniform-critical-random-set}, specifically \cref{eq:critical-scalar-tv-rate}.

Since the approximation holds for the entire null configuration, every measurable statistic of the null $r$-spaces transfers to the independent Bernoulli model at the same total-variation scale. This random-set upgrade uses $\mathrm{GL}_N(q)$-symmetry and suppression of non-direct occupied configurations rather than a direct Stein--Chen argument; compare \cite{BarbourBrown1992}.

For $k=3$, the exact profile in \cref{eq:trilinear-exact-profile} gives the clearest instance of the entropy--stability mechanism. The same analysis also yields the resonant two-point law for the maximum null dimension; see \cref{cor:complete-maximum}. At the opposite end, \cref{prop:tv-lower-obstruction} shows that the approximation cannot converge arbitrarily faster: the first distance-one dependence already appears on the exponential scale
\[
 q^{-\left(\frac{k-1}{k!}+o(1)\right)m r^{k-1}}.
\]
Appendix~\ref{app:quantized-stability} identifies the same constant $(k-1)/k!$ as the first possible exterior-dependence threshold.

\begin{theorem}[Uniform low-excess exterior stability]
\label{thm:stability-intro}
Fix $k\ge2$ and $\Gamma>0$. There exist constants $C_{k,\Gamma},K_{k,\Gamma}>0$, depending only on $k$ and $\Gamma$, such that, over every field, every $\Gamma$-low-excess family $\mathcal H$ admits a partition $\mathcal H=\mathcal H_1\sqcup\cdots\sqcup\mathcal H_\kappa$ and representatives $A_i\in\mathcal H_i$ such that $d_{\mathrm{Gr}}(H,A_i)\le K_{k,\Gamma}$ for every $H\in\mathcal H_i$, while
\[
0\le \kappa r-\dim(A_1+\cdots+A_\kappa)\le C_{k,\Gamma}.
\]
\end{theorem}
The proof is given in \cref{sec:low-excess-stability}, after the localization lemma \cref{lem:bounded-radius-clustering}.

In particular, near equality cannot arise from diffuse weak dependence among many subspaces: dependence is concentrated in uniformly bounded clusters. The stronger normal-form statement proved in \cref{sec:low-excess-stability} sharpens this to bounded-thickness Grassmann clusters: every cluster spans only $r+O(1)$ dimensions, different clusters are separated at scale $r$, and representatives from distinct clusters have only $O(1)$ total ordinary span deficiency.

The underlying endpoint budget is stronger: in any ordering, only $O_{k,\Gamma}(1)$ steps have a nonzero but non-full overlap with the preceding span, while all remaining steps are exactly direct or exactly contained. This control reduces the count of low-excess configurations enough to close the borderline $k=3$ factorial-moment argument; Appendix~\ref{app:quantized-stability} records the corresponding quantized bounds.

The critical balance is related to deterministic isotropy thresholds. Feldman--Propp proved the corresponding lower bound on the ambient dimension required to force an isotropic subspace \cite{FeldmanPropp1992}, and Chen--Xu--Ye determine the associated deterministic isotropy indices for alternating multilinear maps over algebraically closed fields of arbitrary characteristic \cite{ChenXuYe2026}. Those results concern deterministic existence thresholds. Here the map is random and we determine critical points; the main theorem is uniform even when the field size $q$ and target dimension $m$ vary with $r$.

A direct probabilistic precursor is Eberhard--Sabatini \cite[Proposition~4.1(2)]{EberhardSabatini2025}: their Poisson law for isotropic $3$-spaces of a random alternating bilinear map over $\F_p$ is the specialization $k=2$, $r=3$, $m=N-3$ of our high-target theorem (\cref{thm:poisson-general-intro}).

Related generic-isotropy, multilinear-variety, and extremal results include \cite{Tevelev2001,Anzaldo2026,ChenYe2025,ChenYeGeometry2026,Qiao2023,ConlonPohoataZakharov2021}; for broader exterior-algebra and subspace-arrangement methods, see also \cite{GhorpadePatilPillai2009,Kinser2011,ScottWilmer2021}.

The bilinear case appears to be different. For one alternating form the critical count can converge to zero in probability even though its expectation converges to $C_q$; for a pair of forms there is always a common isotropic $r$-space at the critical dimension and the second factorial moment diverges. These obstructions are proved in \cref{sec:bilinear-boundary}.

More generally, \cref{prop:rplusone-space-moment-obstruction} shows that null $(r+1)$-spaces are a universal source of high-moment divergence: they are asymptotically absent at the level of probability, yet when this obstruction applies they force factorial-moment divergence at order $O(m r^{k-2})$. This separates process-level Poisson approximation from Poisson factorial-moment approximation at growing orders and identifies a second scale beyond the stability scale. The constant $(k-1)/k!$ appearing at that second scale also governs the first possible exterior dependence and the total-variation obstruction proved below.

\section{First moments, threshold localization, and pair overlaps}\label{sec:preliminaries}
\subsection{Finite-field estimates}

For $0\le r\le N$, $\qbinom Nr=\prod_{i=0}^{r-1}(q^{N-i}-1)/(q^{r-i}-1)$. For $\ell\ge1$ set $C_{q,\ell}=\prod_{j=1}^{\ell}(1-q^{-j})^{-1}$, so $C_q=C_{q,\infty}$. Since $q\ge2$, one has $C_q\le C_2$. Then
\begin{equation}
 q^{r(N-r)}\le\qbinom Nr\le C_q q^{r(N-r)}. \label{eq:qbinom-bounds}
\end{equation}

\subsection{First moment and arithmetic locking}

Set $\Delta_N(r)=k(N-r)-m\binom{r-1}{k-1}$. Then
\begin{equation}
 r(N-r)-m\binom rk=\frac r k\Delta_N(r), \label{eq:locking-identity}
\end{equation}
\begin{equation}
 \Delta_N(r+1)-\Delta_N(r)=-k-m\binom{r-1}{k-2}<0. \label{eq:delta-step}
\end{equation}

\begin{proposition}[Arithmetic locking]\label{prop:locking}
Fix $q$, $k\ge3$, and $m$. Let $M_N=\max\{1\le r\le N:\Delta_N(r)\ge0\}$.
\begin{enumerate}[label=\textup{(\roman*)}]
 \item along every sequence $N\to\infty$ for which $\Delta_N(M_N)>0$, $\Pp(\alpha_N=M_N)\to1$;
 \item along every sequence $N\to\infty$ for which $\Delta_N(M_N)=0$, $\Pp(\alpha_N\in\{M_N-1,M_N\})\to1$.
\end{enumerate}
\end{proposition}
The proof is deferred until after the pair-overlap estimate \cref{lem:pair-overlap}.

\subsection{Pair-overlap suppression}

For subspaces $A,B\le V$,
\begin{equation}
 \Lambda^kA\cap\Lambda^kB=\Lambda^k(A\cap B).
 \label{eq:exterior-intersection}
\end{equation}
Indeed, choose a basis of $A\cap B$, extend it separately to bases of $A$ and $B$, and then to a basis of $V$; the corresponding wedge basis of $\Lambda^kV$ gives the identity.

For two $r$-spaces $H_1,H_2$, we have $\dim(\Lambda^kH_2+\Lambda^kH_1)=2\binom rk-\binom {d_2}{k}$. For fixed $H_1$, the number of $H_2$ with intersection dimension $d_2$ is
\begin{equation}
 \mathbb{H}_{d_2}=\qbinom {r}{d_2}\qbinom{N-r}{r-d_2}q^{(r-d_2)^2}.
\label{eq:grassmann-intersection-count}
\end{equation}
Let $\pi_{d_2}=\mathbb{H}_{d_2}/\qbinom Nr$. By \cref{eq:qbinom-bounds},
\begin{equation}
 \pi_{d_2}\le C_q^2q^{-d_2(N-2r+d_2)}.
 \label{eq:pt-bound}
\end{equation}

Conditioning on the intersection dimension $d_2=\dim(H_1\cap H_2)$ and using \cref{eq:exterior-intersection} gives the exact pair identity
\begin{equation}
 \frac{\E[X_{N,r}(X_{N,r}-1)]}{\mu_{N,r}^2}
 =
 \sum_{d_2=0}^{r-1}\pi_{d_2}q^{m\binom {d_2}{k}}.
 \label{eq:exact-pair-factorial-moment}
\end{equation}

\begin{lemma}[Pair-overlap bound]\label{lem:pair-overlap}
Uniformly over all prime powers $q$, let $2\le k\le r$ and $m\ge1$ vary, and suppose $N-2r\to\infty$ and
\begin{equation}
 m\binom rk\le r(N-r).
 \label{eq:below-threshold}
\end{equation}
Then
\[
 \sum_{d_2=1}^{r-1}\pi_{d_2}q^{m\binom {d_2}{k}}
 =O\!\left(q^{-(N-2r)/2}\right)=o(1).
\]
\end{lemma}

\begin{proof}
Put $c=N-r$ and $x=d_2/r$. Since $\binom {d_2}{k}/\binom rk\le\left(d_2/r\right)^k$, \cref{eq:below-threshold} gives $m\binom {d_2}{k}\le cd_2\,x^{k-1}$. Hence the exponent in \cref{eq:pt-bound} satisfies
\begin{align*}
 -d_2(c-r+d_2)+m\binom {d_2}{k}
 &\le -cd_2(1-x^{k-1})+rd_2(1-x)\\
 &=-d_2(1-x)\bigl(c(1+x+\cdots+x^{k-2})-r\bigr)\\
 &\le -d_2(1-x)(c-r).
\end{align*}
If $d_2\le r/2$, then $d_2(1-x)\ge d_2/2$; if $d_2\ge r/2$, then $d_2(1-x)\ge(r-d_2)/2$. Summing the two geometric tails and using $C_q\le C_2$ proves the claim with an absolute implied constant.
\end{proof}

Since $q^{m\binom {d_2}{k}}\ge1$, the lemma gives $\sum_{d_2=1}^{r-1}\pi_{d_2}=o(1)$, while $\pi_r=\qbinom Nr^{-1}=o(1)$; hence $\pi_0=1-o(1)$. Combining this with \cref{eq:exact-pair-factorial-moment} gives
\begin{equation}
 \frac{\E[X_{N,r}(X_{N,r}-1)]}{\mu_{N,r}^2}=1+o(1).
 \label{eq:second-moment-consequence}
\end{equation}
If in addition $\mu_{N,r}\to\infty$, then $\operatorname{Var}X_{N,r}=\E[X_{N,r}(X_{N,r}-1)]+\E X_{N,r}-(\E X_{N,r})^2$, so Chebyshev's inequality gives $\Pp(X_{N,r}>0)\to1$.

\begin{proof}[Proof of \cref{prop:locking}]
Since $\binom{r-1}{k-1}=r^{k-1}/(k-1)!+O_k(r^{k-2})$, putting $r=\lfloor aN^{1/(k-1)}\rfloor$ gives
\[
 \frac{\Delta_N(r)}{N}
 =k-\frac{m a^{k-1}}{(k-1)!}+o_{k,m,a}(1).
\]
Choosing $a$ respectively below and above $\bigl(k(k-1)!/m\bigr)^{1/(k-1)}$ brackets the unique crossing. Hence monotonicity from \cref{eq:delta-step} gives $M_N=\Theta_{k,m}(N^{1/(k-1)})=o(N)$. In particular, $M_N\to\infty$ and $N-2M_N\to\infty$.

Suppose $\Delta_N(M_N)>0$. Since $\Delta_N(M_N)$ is an integer, \cref{eq:first-moment,eq:qbinom-bounds,eq:locking-identity} give $\mu_{N,M_N}\ge q^{M_N/k}\to\infty$. The condition $\Delta_N(M_N)\ge0$ is exactly \cref{eq:below-threshold}, so \cref{eq:second-moment-consequence} gives $X_{N,M_N}>0$ with high probability. Maximality of $M_N$ gives $\Delta_N(M_N+1)\le-1$, hence $\mu_{N,M_N+1}\le C_qq^{-(M_N+1)/k}\to0$. Any null subspace of dimension larger than $M_N$ contains one of dimension $M_N+1$, so Markov's inequality yields $\alpha_N=M_N$ with high probability.

Now suppose $\Delta_N(M_N)=0$ and write $r=M_N$. By \cref{eq:delta-step}, $\Delta_N(r-1)=k+m\binom{r-2}{k-2}\ge k$, so $\mu_{N,r-1}\to\infty$ and \cref{eq:second-moment-consequence} gives $X_{N,r-1}>0$ with high probability. Also $\Delta_N(r+1)=-k-m\binom{r-1}{k-2}\le-k$, so $\mu_{N,r+1}\to0$. Therefore $\Pp(\alpha_N\in\{r-1,r\})\to1$.
\end{proof}

\section{Exterior rank geometry}\label{sec:exterior-geometry}
Throughout this section, the ambient field is arbitrary.

For $x\in\Lambda^kV$, define its support by
\[
 \operatorname{supp}(x)
 :=
 \Span\{
 \iota_{\phi_{k-1}}\cdots\iota_{\phi_1}x:
 \phi_1,\ldots,\phi_{k-1}\in V^*
 \},
\]
where $V^*=\operatorname{Hom}(V,\mathbb F)$ is the dual space and $\iota_\phi$ denotes contraction by $\phi$.

\begin{lemma}[Support]\label{lem:support-contraction}
The subspace $\operatorname{supp}(x)$ is the least $W\le V$ such that $x\in\Lambda^kW$.
\end{lemma}

\begin{proof}
If $x\in\Lambda^kU$, every displayed contraction lies in $U$, so $\operatorname{supp}(x)\le U$. Conversely, choose $V=\operatorname{supp}(x)\oplus Z$. If $x$ had a wedge component containing a vector from $Z$, contraction against the remaining $k-1$ factors would produce a vector with nonzero $Z$-component, contrary to the definition of $\operatorname{supp}(x)$.
\end{proof}

\begin{lemma}[Support in a direct sum]\label{lem:support-direct-sum}
Let $k\ge2$ and $V=V_1\oplus\cdots\oplus V_b$. If $x_i\in\Lambda^kV_i$, then $\operatorname{supp}(x_1+\cdots+x_b)=\bigoplus_{i=1}^b\operatorname{supp}(x_i)$.
\end{lemma}

\begin{proof}
Every contraction of $x_1+\cdots+x_b$ lies in the right-hand side. Conversely, extending covectors on $V_i$ by zero on the other summands shows that every contraction of $x_i$ is also a contraction of $x_1+\cdots+x_b$.
\end{proof}

\begin{theorem}[Deficiency-sensitive exterior expansion]\label{thm:exterior-intro}
Let $2\le k\le r$, and write the ordinary span deficiency as $\Phi=tr-\mathfrak O=ar+b$ with $0\le b<r$. Then
\begin{equation}
 \mathfrak E
 \ge (t-a)\binom rk-\binom bk. \label{eq:deficiency-exterior-expansion}
\end{equation}
If $\Phi>0$ and $b=0$, then
\begin{equation}
 \mathfrak E
 \ge (t-a)\binom rk+\binom{r-1}{k-1}. \label{eq:multiple-deficiency-gain}
\end{equation}
Consequently
\begin{equation}
 \mathfrak E\ge\beta \mathfrak O,
 \label{eq:exterior-expansion}
\end{equation}
and, whenever the ordinary sum is not direct,
\begin{equation}
 \mathfrak E\ge\beta\left(\mathfrak O+1\right).
 \label{eq:strict-exterior-expansion}
\end{equation}
Equality in \cref{eq:exterior-expansion} holds exactly for an ordinary direct sum, and the normalized one-unit gap in \cref{eq:strict-exterior-expansion} is sharp.
\end{theorem}
The proof follows the overlap-profile estimate \cref{lem:overlap-profile}.

For $(r,k)=(3,2)$, \cref{eq:exterior-expansion} becomes $\mathfrak E\ge \mathfrak O$, recovering the rank inequality used by Eberhard--Sabatini \cite[Lemma~4.2]{EberhardSabatini2025} in their critical bilinear model. The form above retains the full residue of the ordinary span deficiency and, when that deficiency is a positive multiple of $r$, gives the stronger gain in \cref{eq:multiple-deficiency-gain}.

\begin{lemma}[Overlap-profile inequality]\label{lem:overlap-profile}
\begin{equation}
 tr-\mathfrak O=\sum_{j=2}^t d_j,
 \qquad
 t\binom rk-\mathfrak E=\sum_{j=2}^t e_j,
 \label{eq:profile-identities}
\end{equation}
and
\begin{equation}
 e_j\le\binom{d_j}{k},
 \qquad
 \mathfrak E\ge t\binom rk-\sum_{j=2}^t\binom{d_j}{k}.
 \label{eq:profile-rank-bound}
\end{equation}
\end{lemma}

\begin{proof}
The identities in \cref{eq:profile-identities} follow by summing the ordinary and exterior dimension increments. Since $\Lambda^kH_1+\cdots+\Lambda^kH_{j-1} \le \Lambda^k(H_1+\cdots+H_{j-1})$, \cref{eq:exterior-intersection} gives
\[
 \Lambda^kH_j\cap(\Lambda^kH_1+\cdots+\Lambda^kH_{j-1})
 \subseteq
 \Lambda^k\bigl(H_j\cap(H_1+\cdots+H_{j-1})\bigr),
\]
which proves $e_j\le\binom{d_j}{k}$ and hence \cref{eq:profile-rank-bound}.
\end{proof}

\begin{proof}[Proof of \cref{thm:exterior-intro}]
Write $\Phi=ar+b$ with $0\le b<r$. We use the elementary convexity of $d\mapsto\binom dk$: since $\binom{d+1}{k}-\binom dk=\binom d{k-1}$ is increasing in $d$, concentrating a fixed total mass can only increase the sum (e.g., $\binom32+\binom32<\binom42+\binom22$). Thus, whenever $0\le d_i\le r$ and $\sum_i d_i=ar+b$,
\[
 \sum_i\binom{d_i}{k}\le a\binom rk+\binom bk.
\]
If $b=0$ and at least one positive entry is strictly smaller than $r$, the largest possible sum is instead $(a-1)\binom rk+\binom{r-1}{k}=a\binom rk-\binom{r-1}{k-1}$, attained by the profile $(r,\ldots,r,r-1,1,0,\ldots,0)$.

Apply this observation to the overlaps in \cref{lem:overlap-profile}. Since $\sum_{j=2}^t d_j=\Phi$, we obtain $\mathfrak E\ge (t-a)\binom rk-\binom bk$, proving \cref{eq:deficiency-exterior-expansion}.

Suppose now that $\Phi>0$ and $b=0$. We choose an ordering that exposes one unavoidable non-extremal step. If some pair intersects nontrivially, put that pair first. Distinctness gives $1\le d_2\le r-1$, so the strict form of the preceding convexity bound and \cref{lem:overlap-profile} give
\[
 \mathfrak E
 \ge (t-a)\binom rk+\binom{r-1}{k-1}.
\]

It remains to consider a pairwise-disjoint family. Choose an inclusion-minimal subfamily whose ordinary sum is not direct and order it first, say $H_1,\ldots,H_\ell$. Then $H_1\oplus\cdots\oplus H_{\ell-1}$ is direct and $d_\ell>0$. Moreover $\Lambda^kH_\ell\cap\bigoplus_{i<\ell}\Lambda^kH_i=0$. Indeed, if $x=\sum_{i<\ell}x_i$ belongs to the intersection, then \cref{lem:support-direct-sum} gives
\[
 \operatorname{supp}(x)=\bigoplus_{i<\ell}\operatorname{supp}(x_i)\subseteq H_\ell.
\]
Each summand support therefore lies in $H_i\cap H_\ell=0$, so every $x_i=0$. Thus $e_\ell=0$.

If $d_\ell=r$, the remaining defect is $(a-1)r$, and the same convexity bound gives at most $(a-1)\binom rk$ for the remaining exterior overlaps. If $1\le d_\ell<r$, their total ordinary defect is $ar-d_\ell=(a-1)r+(r-d_\ell)$, so their exterior overlaps are at most
\[
 (a-1)\binom rk+\binom{r-d_\ell}{k}
 \le (a-1)\binom rk+\binom{r-1}{k}.
\]
In both cases $\sum_{j=2}^t e_j\le a\binom rk-\binom{r-1}{k-1}$, which proves \cref{eq:multiple-deficiency-gain}.

To derive the normalized form, note that $\beta \mathfrak O=(t-a)\binom rk-\beta b$. If $1\le b<r$, \cref{eq:deficiency-exterior-expansion} gives $\mathfrak E-\beta \mathfrak O\ge\beta b-\binom bk$. The function $g(d)=\beta d-\binom dk$ has nonincreasing first differences $g(d+1)-g(d)=\beta-\binom d{k-1}$. Hence it is discretely concave on $1\le d\le r-1$ and its minimum occurs at an endpoint; here $g(1)=\beta$ and $g(r-1)=(k-1)\beta$. Thus \cref{eq:strict-exterior-expansion} holds. If $b=0<\Phi$, then $\mathfrak E-\beta \mathfrak O\ge\binom{r-1}{k-1}=k\beta\ge\beta$. This proves \cref{eq:exterior-expansion,eq:strict-exterior-expansion}; equality in \cref{eq:exterior-expansion} is exactly the direct-sum case. Finally, if two $r$-spaces meet in a line, then $\mathfrak O=2r-1$ and \cref{eq:exterior-intersection} gives $\mathfrak E=2\binom rk$, so the excess above the right side of \cref{eq:exterior-expansion} is exactly $\beta$. Thus the normalized one-unit gap is sharp.
\end{proof}

\section{Low-excess stability and Grassmann clusters}\label{sec:low-excess-stability}
Throughout this section, $k\ge2$ and $\Gamma>0$ are fixed, the field is arbitrary, and all bounds are uniform in the field and $|\mathcal H|$.

The argument has two stages: an endpoint estimate first reduces every step to either a new block or an attachment, and a localization argument then groups attachments around earlier new blocks.

\subsection{Endpoint dichotomy}
For every ordering,
\begin{equation}
 G_k(\mathcal H)=\sum_{j=2}^t(\beta d_j-e_j). \label{eq:fixed-excess-sum}
\end{equation}

Every summand is nonnegative because $e_j\le\binom{d_j}{k}\le\beta d_j$. Since $G_k$ is independent of the ordering, every subfamily $\mathcal H'\subseteq\mathcal H$ also satisfies $0\le G_k(\mathcal H')\le G_k(\mathcal H)$. Indeed, reorder the family so that the members of $\mathcal H'$ come first; the sum over that initial block is exactly the excess of the subfamily.

\begin{lemma}[Global endpoint-defect budget]\label{lem:low-excess-endpoints}
There exists $C_{k,\Gamma}>0$, depending only on $k$ and $\Gamma$, such that, with $D_{k,\Gamma}:=\lceil C_{k,\Gamma}\rceil$, every ordering of a $\Gamma$-low-excess family $\mathcal H$ satisfies
\begin{equation}
 \sum_{j=2}^t\min\{d_j,r-d_j\}\le C_{k,\Gamma}. \label{eq:global-endpoint-defect}
\end{equation}
In particular, every $j\ge2$ satisfies $d_j\le D_{k,\Gamma}$ or $r-d_j\le D_{k,\Gamma}$, and
\begin{equation}
 \#\{j\ge2:0<d_j<r\}\le \lfloor C_{k,\Gamma}\rfloor. \label{eq:nonendpoint-step-count}
\end{equation}
Thus, apart from at most $\lfloor C_{k,\Gamma}\rfloor$ indices, every step is an exact endpoint: $d_j=0$ or $d_j=r$.
\end{lemma}

\begin{proof}
Set $f_{r,k}(0)=0$. For $1\le d\le r$ define
\[
 f_{r,k}(d)=\beta d-\binom dk
 =\frac d k\left(\binom{r-1}{k-1}-\binom{d-1}{k-1}\right).
\] By \cref{eq:fixed-excess-sum,lem:overlap-profile}, $\beta d_j-e_j\ge f_{r,k}(d_j)\ge0$. If $1\le d\le r/2$, then $\binom{d-1}{k-1}/\binom{r-1}{k-1}\le2^{1-k}$, and hence
\[
 f_{r,k}(d)
 \ge \frac d k(1-2^{1-k})\binom{r-1}{k-1}
 \ge \eta_{1,k}\,d\,r^{k-1}
\]
for some $\eta_{1,k}>0$. If $r/2\le d<r$ and $s=r-d$, Pascal's identity gives
\[
 \binom{r-1}{k-1}-\binom{d-1}{k-1}
 =\sum_{h=d-1}^{r-2}\binom h{k-2},
\]
so
\[
 f_{r,k}(d)
 \ge \frac r{2k}\,s\binom{\lfloor r/2\rfloor-1}{k-2}
 \ge \eta_{2,k}\,s\,r^{k-1}
\]
for some $\eta_{2,k}>0$. At $d=0$ and $d=r$ both sides below vanish. Consequently, with $\tau_k=\min\{\eta_{1,k},\eta_{2,k}\}>0$,
\begin{equation}
 f_{r,k}(d)\ge
 \tau_k r^{k-1}\min\{d,r-d\}
 \qquad(0\le d\le r).
\label{eq:uniform-endpoint-profile}
\end{equation}
Summing over $j$ and using \cref{eq:fixed-excess-sum}, $\Gamma r^{k-1}\ge G_k(\mathcal H) \ge \tau_k r^{k-1}\sum_{j=2}^t\min\{d_j,r-d_j\}$. Thus \cref{eq:global-endpoint-defect} holds with the fixed choice $C_{k,\Gamma}:=\Gamma/\tau_k$ and $D_{k,\Gamma}:=\lceil C_{k,\Gamma}\rceil$.
\end{proof}

Fix these choices of $C_{k,\Gamma}$ and $D_{k,\Gamma}$ from now on.

Take $r>2D_{k,\Gamma}$. Since $\dim(\mathcal O_{<j}+H_j)-\dim\mathcal O_{<j}=r-d_j$, the quantity $d_j=\dim(H_j\cap\mathcal O_{<j})$ measures the overlap with the preceding span, while $r-d_j$ is exactly the number of new ordinary dimensions contributed by $H_j$.

Call $H_j$ a \emph{new-block step} if $d_j\le D_{k,\Gamma}$ and an \emph{attachment step} if $r-d_j\le D_{k,\Gamma}$. By \cref{lem:low-excess-endpoints}, these alternatives are disjoint and exhaustive. A new-block step has small overlap and adds at least $r-D_{k,\Gamma}$ dimensions; an attachment is almost contained in the preceding span and adds at most $D_{k,\Gamma}$ dimensions. At the exact endpoints, $d_j=0$ means direct from the preceding span, while $d_j=r$ means contained in it.

\begin{figure}[H]
\centering
\begin{tikzpicture}
\begin{axis}[
 width=0.82\textwidth,
 height=5.8cm,
 xmin=0,xmax=1,
 ymin=0,ymax=0.42,
 axis lines=left,
 xlabel={$x=d/r$},
 ylabel={$x-x^3$},
 xtick={0,0.5,1},
 xticklabels={$0$,$1/2$,$1$},
 ytick=\empty,
 domain=0:1,
 samples=180,
 clip=false
]
 \addplot[black,very thick] {x-x^3};
 \node[anchor=south west,align=left] at (axis cs:0.145,0.058)
  {\scriptsize $d=O(1)$\\[-1pt]\scriptsize new block};
 \node[anchor=south east,align=right] at (axis cs:0.885,0.058)
  {\scriptsize $r-d=O(1)$\\[-1pt]\scriptsize attachment};
 \node[align=center,fill=white,inner sep=1.5pt] at (axis cs:0.52,0.235)
  {\scriptsize intermediate overlap\\[-1pt]\scriptsize costs $\Theta(r^3)$};
 \draw[-{Latex[length=1.8mm]},thin]
  (axis cs:0.14,0.07) -- (axis cs:0.025,0.01);
 \draw[-{Latex[length=1.8mm]},thin]
  (axis cs:0.89,0.08) -- (axis cs:0.975,0.01);
\end{axis}
\end{tikzpicture}
\caption{One-step endpoint dichotomy for $k=3$. At a single step $j$, for $d_r=\lfloor xr\rfloor$, $6r^{-3}f_{r,3}(d_r)\longrightarrow x-x^3$. Low excess leaves two possible step types: $d_j=O(1)$ gives a new-block step, while $r-d_j=O(1)$ gives an attachment step. Intermediate overlap costs order $r^3$, exceeding the $O(r^2)$ low-excess budget, and is therefore excluded; see \cref{lem:low-excess-endpoints}.}
\label{fig:exterior-excess-landscape}
\end{figure}
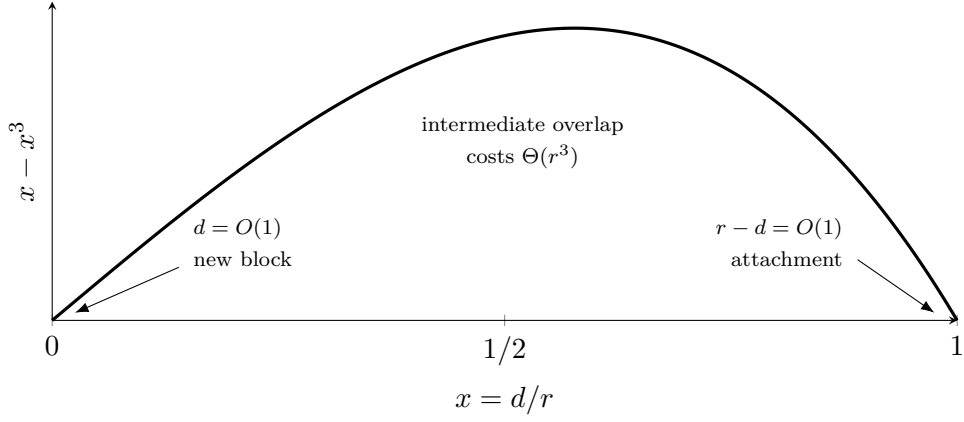
More generally, for fixed $k$ and $d_r=\lfloor xr\rfloor$, $k!r^{-k}f_{r,k}(d_r)\to x-x^k$.

By \cref{eq:nonendpoint-step-count}, all but $O_{k,\Gamma}(1)$ steps are exact endpoints. The endpoint dichotomy alone does not yet produce clusters: an attachment is only known to be almost contained in the whole preceding span. The remaining task is to show that it is actually close to one earlier new-block space.

\subsection{Localization and canonical blocks}
The next lemma isolates the pure exterior contribution of the earlier new-block spaces; the capped convexity bound then forces an attachment to concentrate near one of them.

\begin{lemma}[Pure-block intersection]\label{lem:pure-block-intersection}
Let $V_1,\ldots,V_\ell$ have direct sum and let $L$ be any subspace of the ambient vector space. Then, for every $k\ge2$, $\Lambda^kL\cap\bigoplus_{i=1}^{\ell}\Lambda^kV_i=\bigoplus_{i=1}^{\ell}\Lambda^k(L\cap V_i)$.
\end{lemma}

\begin{proof}
The inclusion from right to left is immediate. Conversely, write $x=x_1+\cdots+x_\ell$ with $x_i\in\Lambda^kV_i$ and $x\in\Lambda^kL$. By \cref{lem:support-direct-sum}, $\bigoplus_i\operatorname{supp}(x_i)=\operatorname{supp}(x)\le L$. Hence $\operatorname{supp}(x_i)\le L\cap V_i$ for every $i$, so \cref{lem:support-contraction} gives $x_i\in\Lambda^k(L\cap V_i)$.
\end{proof}

We use the following elementary discrete-convexity bound. If $0\le L\le r/2$, $0\le a_i\le r-L$, and $\sum_i a_i\le r$, then
\begin{equation}
 \sum_i\binom{a_i}{k}\le \binom{r-L}{k}+\binom Lk,
 \label{eq:capped-convex-concentration}
\end{equation}
with the second term omitted when there is only one $a_i$. Indeed, transfer mass from smaller positive coordinates to larger ones until the cap $r-L$ is reached; convexity of $d\mapsto\binom dk$ leaves the extremal profile $(r-L,L,0,\ldots)$.

\begin{lemma}[Uniform localization of attachment steps]\label{lem:bounded-radius-clustering}
There exists $K_{k,\Gamma}>D_{k,\Gamma}$, depending only on $k$ and $\Gamma$, such that every ordering of a $\Gamma$-low-excess family $\mathcal H$ has the following property: every attachment step $H_j$ satisfies $d_{\mathrm{Gr}}(H_j,H_i)\le K_{k,\Gamma}$ for some earlier new-block step $H_i$.
\end{lemma}

\begin{proof}
Fix an attachment step $H_j$, and let the earlier new-block spaces be $A_1,\ldots,A_\ell$. For the prefix $H_1,\ldots,H_{j-1}$, \cref{eq:global-endpoint-defect} gives
\[
 \sum_{\substack{2\le h<j\\ H_h\text{ new-block}}}d_h
 +
 \sum_{\substack{2\le h<j\\ H_h\text{ attachment}}}(r-d_h)
 \le C_{k,\Gamma}.
\]
There are $\ell$ new-block steps in the prefix, including $H_1$, so the ordinary span increment formula becomes
\[
 \dim \mathcal O_{<j}
 =\ell r
 -\sum_{\substack{2\le h<j\\ H_h\text{ new-block}}}d_h
 +\sum_{\substack{2\le h<j\\ H_h\text{ attachment}}}(r-d_h).
\]
Hence
\[
 \dim \mathcal O_{<j}=\ell r+O_{k,\Gamma}(1).
\]
Also the prefix excess is at most the full excess, because \cref{eq:fixed-excess-sum} is a sum of nonnegative terms. Therefore
\begin{equation}
 \dim \mathcal E_{<j}
 =\beta\dim \mathcal O_{<j}+O_{k,\Gamma}(r^{k-1})
 =\ell\binom rk+O_{k,\Gamma}(r^{k-1}). \label{eq:prefix-exterior-newblocks}
\end{equation}

Proceed through the earlier new-block spaces in their order. Choose $A_i^\circ\le A_i$ complementary in $A_i$ to $A_i\cap(A_1+\cdots+A_{i-1})$. Then the sum of the $A_i^\circ$ is direct. Put $c_i=\operatorname{codim}_{A_i}A_i^\circ$. Since $c_i$ is at most the ordinary overlap at the step at which $A_i$ was created, $\sum_{i=1}^{\ell}c_i\le C_{k,\Gamma}$, and set $\mathcal T_0:=\bigoplus_{i=1}^{\ell}\Lambda^kA_i^\circ\le\mathcal E_{<j}$. Using Pascal's identity,
\[
 0\le
 \ell\binom rk-\dim \mathcal T_0
 =\sum_{i=1}^{\ell}\left(
 \binom rk-\binom{r-c_i}{k}\right)
 \le
 \binom{r-1}{k-1}\sum_i c_i
 =O_{k,\Gamma}(r^{k-1}).
\]
Together with \cref{eq:prefix-exterior-newblocks}, this yields
\begin{equation}
 \dim(\mathcal E_{<j}/\mathcal T_0)\le C_0r^{k-1} \label{eq:previous-span-close-to-pure-blocks}
\end{equation}
for a constant $C_0=C_0(k,\Gamma)$ independent of $t$.

Put $L_j=H_j\cap \mathcal O_{<j}$ and $\dim L_j=d_j\ge r-D_{k,\Gamma}$. Since $\mathcal E_{<j}\le\Lambda^k\mathcal O_{<j}$, \cref{eq:exterior-intersection} gives $\Lambda^kH_j\cap \mathcal E_{<j}=\Lambda^kL_j\cap \mathcal E_{<j}$. Restricting the quotient map $\mathcal E_{<j}\to \mathcal E_{<j}/\mathcal T_0$ to $(\Lambda^kL_j)\cap \mathcal E_{<j}$ and using \cref{eq:previous-span-close-to-pure-blocks} gives $\dim(\Lambda^kL_j\cap \mathcal T_0)\ge e_j-C_0r^{k-1}$. The local excess is at most the total excess, so $e_j\ge\beta d_j-\Gamma r^{k-1}$. Hence for a constant $C_1=C_1(k,\Gamma)$,
\begin{equation}
 \dim(\Lambda^kL_j\cap \mathcal T_0)
 \ge \beta d_j-C_1r^{k-1}. \label{eq:large-pure-block-overlap}
\end{equation}

By \cref{lem:pure-block-intersection}, if $a_i=\dim(L_j\cap A_i^\circ)$, then the left side equals
\[
 \sum_{i=1}^{\ell}\binom{a_i}{k},
 \qquad
 \sum_i a_i\le d_j\le r.
\]
Choose the fixed $K_{k,\Gamma}>D_{k,\Gamma}$ so large that $K_{k,\Gamma}/(k-1)!-D_{k,\Gamma}/k!>C_1+1$. For large $r$, also $r>2K_{k,\Gamma}$. If every $a_i\le r-K_{k,\Gamma}$, then \cref{eq:capped-convex-concentration} with $L=K_{k,\Gamma}$ gives
\[
 \sum_i\binom{a_i}{k}\le\binom{r-K_{k,\Gamma}}{k}+\binom{K_{k,\Gamma}}k;
\]
when $\ell=1$, the stronger bound with only the first term holds. On the other hand,
\[
 \beta(r-D_{k,\Gamma})-\binom{r-K_{k,\Gamma}}{k}-\binom{K_{k,\Gamma}}k
 =
 \left(\frac{K_{k,\Gamma}}{(k-1)!}-\frac{D_{k,\Gamma}}{k!}\right)r^{k-1}
 +O_{k,\Gamma}(r^{k-2}),
\]
which contradicts \cref{eq:large-pure-block-overlap}. Thus some $a_i\ge r-K_{k,\Gamma}$, and therefore
\[
 d_{\mathrm{Gr}}(H_j,A_i)
 =r-\dim(H_j\cap A_i)
 \le K_{k,\Gamma}.
\]
\end{proof}

Fix this choice of $K_{k,\Gamma}$ from now on.

\begin{proof}[Proof of \cref{thm:stability-intro}]
Apply \cref{lem:low-excess-endpoints}. Treat $H_1$ as the first new-block step and classify every later step by $d_j\le D_{k,\Gamma}$ or $r-d_j\le D_{k,\Gamma}$. By \cref{lem:bounded-radius-clustering}, every attachment step lies within Grassmann distance at most $K_{k,\Gamma}$ of an earlier new-block space. Assign each attachment to one such space and take the new-block spaces, in their order of appearance, as the representatives $A_1,\ldots,A_\kappa$. At a new-block step the overlap with the span of the earlier representatives is at most the corresponding $d_j$, so \cref{eq:global-endpoint-defect} gives $0\le \kappa r-\dim(A_1+\cdots+A_\kappa) \le \sum_{\substack{j\ge2\\ H_j\text{ new}}} d_j \le C_{k,\Gamma}$.
\end{proof}

\begin{corollary}[Uniform low-excess normal form]\label{cor:low-excess-normal-form-intro}
Under the hypotheses of \cref{thm:stability-intro}, let $[t]=P_1\sqcup\cdots\sqcup P_\kappa$ be the connected-component partition of the graph joining $i\ne j$ whenever $d_{\mathrm{Gr}}(H_i,H_j)<r/2$, and write $\mathcal H_a=\{H_i:i\in P_a\}$. We have:
\begin{enumerate}[label=\textup{(\roman*)}]
 \item each block has a localization representative $A_a^\ast\in\mathcal H_a$ with $d_{\mathrm{Gr}}(H,A_a^\ast)\le K_{k,\Gamma}$ for every $H\in\mathcal H_a$. Hence $d_{\mathrm{Gr}}(H_i,H_j)\le2K_{k,\Gamma}$ within one block, while  $d_{\mathrm{Gr}}(H_i,H_j)\ge r-D_{k,\Gamma}-2K_{k,\Gamma}$ across distinct blocks;
 \item the cluster top $W_a:=\sum_{i\in P_a}H_i$ satisfies $r\le\dim W_a\le r+C_{k,\Gamma}$;
 \item for every choice of representatives $A_a\in\mathcal H_a$, $0\le \kappa r-\dim(A_1+\cdots+A_\kappa)\le C_{k,\Gamma}$;
 \item
  $\mathfrak O=\kappa r+O_{k,\Gamma}(1)$ and $\mathfrak E=\kappa\binom rk+O_{k,\Gamma}(r^{k-1})$.
\end{enumerate}
Thus low exterior excess forces intrinsic Grassmann clusters of bounded diameter and top thickness, separated at scale $r$. The representative span deficiency in \textup{(iii)} is an ordinary span deficiency, distinct from the exterior excess $G_k$.
\end{corollary}

\begin{figure}[H]
\centering
\begin{tikzpicture}[
 x=1cm,y=1cm,
 every node/.style={font=\scriptsize},
 rep/.style={circle,draw=black,fill=white,inner sep=2.3pt,line width=0.8pt},
 pt/.style={circle,fill=black,inner sep=1.4pt}
]
\coordinate (Ca) at (2.3,0);
\coordinate (Cb) at (7.5,0);
\draw[black!60,dashed,line width=0.8pt] (Ca) circle (1.55);
\draw[black!60,dashed,line width=0.8pt] (Cb) circle (1.55);
\node[above=1.62cm] at (Ca) {$\mathcal H_a\subseteq B_{\mathrm{Gr}}(A_a^\ast,K_{k,\Gamma})$};
\node[above=1.62cm] at (Cb) {$\mathcal H_b\subseteq B_{\mathrm{Gr}}(A_b^\ast,K_{k,\Gamma})$};

\node[rep] (Aa) at (Ca) {};
\node[below=2pt of Aa] {$A_a^\ast$};
\node[pt] (a1) at (1.25,0.55) {};
\node[pt] (a2) at (3.25,0.55) {};
\node[pt] (a3) at (2.95,-0.75) {};
\node[pt] (a4) at (1.45,-0.70) {};

\node[rep] (Ab) at (Cb) {};
\node[below=2pt of Ab] {$A_b^\ast$};
\node[pt] (b1) at (6.45,0.55) {};
\node[pt] (b2) at (8.45,0.55) {};
\node[pt] (b3) at (8.15,-0.75) {};
\node[pt] (b4) at (6.65,-0.70) {};

\draw[-{Latex[length=1.4mm]},thin] (Aa)--(a3);
\node[right=1pt] at ($(Aa)!0.55!(a3)$) {$\le K_{k,\Gamma}$};
\draw[{Latex[length=1.4mm]}-{Latex[length=1.4mm]},thin] (a1)--(a2);
\node[above=1pt] at ($(a1)!0.5!(a2)$) {$\le2K_{k,\Gamma}$};
\draw[{Latex[length=1.4mm]}-{Latex[length=1.4mm]},thin] (a2)--(b1);
\node[above=2pt] at ($(a2)!0.5!(b1)$) {$\ge r-D_{k,\Gamma}-2K_{k,\Gamma}$};

\node[align=center] at (2.3,-2.15) {$W_a=\sum_{i\in P_a}H_i$\\$0\le\dim W_a-r\le C_{k,\Gamma}$};
\node[align=center] at (7.5,-2.15) {$W_b=\sum_{i\in P_b}H_i$\\$0\le\dim W_b-r\le C_{k,\Gamma}$};
\end{tikzpicture}
\caption{Two canonical Grassmann blocks from \cref{cor:low-excess-normal-form-intro}, illustrating localization, within-block diameter, and cross-block separation.}
\label{fig:grassmann-blocks}
\end{figure}
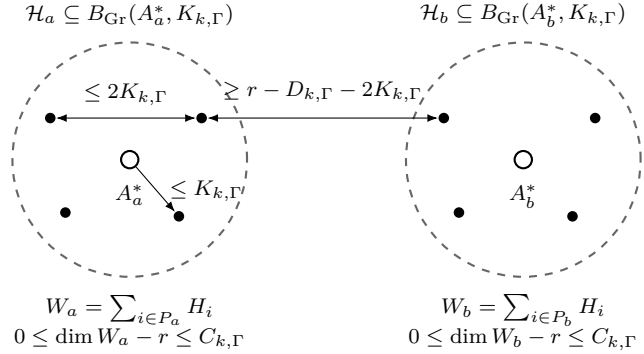

\begin{proof}[Proof of \cref{cor:low-excess-normal-form-intro}]
By \cref{lem:low-excess-endpoints,lem:bounded-radius-clustering}, every attachment lies within distance $K_{k,\Gamma}$ of an earlier new-block space. Distinct new-block spaces $A,B$ satisfy $d_{\mathrm{Gr}}(A,B)\ge r-D_{k,\Gamma}$, so for $r>D_{k,\Gamma}+2K_{k,\Gamma}$ this new-block space is unique. Its assigned block therefore lies in the Grassmann ball $B_{\mathrm{Gr}}(A,K_{k,\Gamma})$, has diameter at most $2K_{k,\Gamma}$, and is at distance at least $r-D_{k,\Gamma}-2K_{k,\Gamma}$ from every other block. Since, for large $r$, $2K_{k,\Gamma}<r/2<r-D_{k,\Gamma}-2K_{k,\Gamma}$ these assigned blocks are exactly the connected components defining the $P_a$, proving \textup{(i)}. In any ordering, the first member of a block is a new-block step, while every later member is an attachment because it meets the first member in dimension at least $r-2K_{k,\Gamma}>D_{k,\Gamma}$.

Fix a block $P_a$ and order that subfamily with any one of its members first. By heredity of the exterior excess, its excess is at most $\Gamma r^{k-1}$. Every later member has ordinary overlap at least $r-2K_{k,\Gamma}$ with the preceding span, so for large $r$ its endpoint defect is $r-d_j$. The global budget gives $\dim\sum_{i\in P_a}H_i=r+\sum_{j\ge2}(r-d_j)\le r+C_{k,\Gamma}$. This proves part \textup{(ii)}.

Now choose one representative $A_a$ from each block. The representative subfamily again has excess at most $\Gamma r^{k-1}$. In any ordering, no later representative can be an attachment step, because localization would put it within distance $K_{k,\Gamma}$ of an earlier new-block representative from a different block, contradicting the cross-block distance bound. Thus all representatives are new-block steps. By \cref{eq:global-endpoint-defect}, $\kappa r-\dim(A_1+\cdots+A_\kappa)=\sum_{a=2}^{\kappa}d_a \le C_{k,\Gamma}$, proving part \textup{(iii)}. Finally, in any ordering there are exactly $\kappa$ new-block steps. Writing the span increment formula separately over new-block and attachment steps, $\dim(H_1+\cdots+H_t)=\kappa r-\sum_{\substack{j\ge2\\ H_j\text{ new}}}d_j+\sum_{\substack{j\ge2\\ H_j\text{ attach}}}(r-d_j)$. The absolute value of the error from $\kappa r$ is at most $C_{k,\Gamma}$ by \cref{eq:global-endpoint-defect}. Hence $\mathfrak O=\kappa r+O_{k,\Gamma}(1)$. Since $\mathfrak E=\beta \mathfrak O+G_k(\mathcal H)$, we obtain $\mathfrak E=\kappa\binom rk+O_{k,\Gamma}(r^{k-1})$. This proves part \textup{(iv)} and completes the normal form.
\end{proof}

\begin{remark}[A limitation of the uniform normal form]
The independence of the constants from $t$ is nonvacuous, and one cannot uniformly require all members of a cluster to contain a common subspace of dimension $r-O(1)$. Over $\F_q$, let $W$ be an $(r+1)$-space and take as the family all $r$-dimensional hyperplanes of $W$. Their ordinary span is $W$, while $\sum_{H\in\Gr(r,W)}\Lambda^kH=\Lambda^kW$ because every decomposable $k$-vector is supported on some hyperplane. Hence the exterior excess is
\[
 \binom{r+1}{k}
 -
 \frac{r+1}{r}\binom rk
 =\Theta_k(r^{k-1}).
\]
The family contains $\qbinom{r+1}{r}\asymp_q q^r$ members, yet the intersection of all its members is $0$. Thus the $r+O(1)$-dimensional cluster top in \cref{cor:low-excess-normal-form-intro} cannot in general be replaced by a common $(r-O(1))$-dimensional subspace contained in every member of the cluster.
\end{remark}

The sharp quantized refinements of these fixed-$k$, fixed-$\Gamma$ bounds are given in Appendix~\ref{app:quantized-stability}.

\section{Poisson approximation in the high-target regime}\label{sec:factorial-moments}

Here the extra $q^{-c}$ factor for non-direct tuples already dominates the counting overhead, so the stability theory of \cref{sec:low-excess-stability} is not needed. For an integer-valued random variable $X$ and $t\ge1$, write $(X)_t:=X(X-1)\cdots(X-t+1)$ for the falling factorial. Poisson convergence will be proved by showing that these factorial moments converge.

\begin{theorem}[High-target resonant Poisson law]\label{thm:poisson-general-intro}
Fix $q$. At resonance, consider any sequence of integers $2\le k\le r$ and $m\ge1$ for which $c/{r^2}\longrightarrow\infty$. Assume $r\to r_\ast\in\{2,3,\ldots,\infty\}$. Then $X_{N,r}\xrightarrow{d}\Pois(C_{q,r_\ast})$, $\,\Pp(\alpha_N=r)\longrightarrow1-e^{-C_{q,r_\ast}}$, and $\,\Pp(\alpha_N=r-1)\longrightarrow e^{-C_{q,r_\ast}}$.
\end{theorem}
The proof is deferred until after the factorial-moment estimate \cref{prop:factorial-error}.

For resonant parameters, write $X_r=X_{N,r}$ and $\mu_r=\mu_{N,r}$. At resonance, reindexing the Gaussian-binomial product gives
\begin{equation}
 \mu_r
 =\frac{\prod_{j=c+1}^{c+r}(1-q^{-j})}
 {\prod_{j=1}^{r}(1-q^{-j})}.
 \label{eq:mean-resonance}
\end{equation}
Consequently, if $c\to\infty$ and $r\to r_\ast\in\{2,3,\ldots,\infty\}$, then $\mu_r\to C_{q,r_\ast}$; if $r\to\infty$, more precisely $\mu_r=C_q+O(q^{-r}+q^{-c})$. Indeed, the numerator in \cref{eq:mean-resonance} is $1+O(q^{-c})$ uniformly in both $q$ and $r$, and $C_{q,r}=C_q+O(q^{-r})$, since $C_q\le C_2$.

\begin{proposition}[Uniform factorial-moment error]\label{prop:factorial-error}
There is an absolute constant $A_{\mathrm{fac}}\ge1$ such that, uniformly over all prime powers $q$, all resonant parameters $2\le k\le r$, $m\ge1$, and all integers $t\ge1$ with $tr\le N$,
\[
 \left|\E(X_r)_t-\mu_r^t\right|
 \le
 A_{\mathrm{fac}}^t t\left(
 q^{-c+(t-1)r}
 +r q^{-c+(t-1)^2r^2/4}
 \right).
\]
\end{proposition}

\begin{proof}
Take the fixed $A_{\mathrm{fac}}$ large enough to dominate all absolute constants in this proof. Choose $r$-spaces independently and uniformly from $\Gr(r,V_N)$. If $H_1+\cdots+H_j$ is direct, then it has dimension $jr$, and the probability that $H_{j+1}$ meets this span nontrivially is at most
\[
 \frac{q^{jr}-1}{q-1}\frac{q^r-1}{q^N-1}
 =O(q^{jr-c}).
\]
A union bound over $1\le j<t$ therefore shows that the proportion of ordered direct $t$-tuples differs from $1$ by at most $O\!\left(tq^{-c+(t-1)r}\right)$. Since $\mu_r\le C_q$, their contribution differs from $\mu_r^t$ by at most $A_{\mathrm{fac}}^t t q^{-c+(t-1)r}$.

Consequently,
\begin{equation}
\Pp\!\left(H_1+\cdots+H_t\text{ is not direct}\right)
\le 4tq^{-c+(t-1)r}.
\label{eq:direct-tuple-deficit}
\end{equation}
For any fixed ordinary span dimension $\mathfrak O$, the number of ordered $t$-tuples with this span dimension is at most
\begin{equation}
 \qbinom {N}{\mathfrak O}\qbinom {\mathfrak O}{r}^t
 \le
 C_q^{t+1}q^{\mathfrak O(N-\mathfrak O)+tr(\mathfrak O-r)}.
 \label{eq:span-tuple-count}
\end{equation}
For non-direct tuples, \cref{thm:exterior-intro} gives joint exterior rank at least $\beta(\mathfrak O+1)$. Hence, using $m\beta=c$, their total contribution is at most
\[
 C_q^{t+1}
 q^{\mathfrak O(N-\mathfrak O)+tr(\mathfrak O-r)-c\mathfrak O-c}
 =
 C_q^{t+1}q^{(\mathfrak O-r)(tr-\mathfrak O)-c}
 \le
 C_q^{t+1}q^{-c+(t-1)^2r^2/4}.
\]
There are at most $tr$ values of $\mathfrak O$. Using $C_q\le C_2$ and the fixed choice of $A_{\mathrm{fac}}$ proves the claim.
\end{proof}

\begin{proof}[Proof of \cref{thm:poisson-general-intro}]
For every fixed $t$, \cref{prop:factorial-error} and $c/r^2\to\infty$ give $\E(X_{N,r})_t-\mu_r^t\to0$. The hypothesis also implies $c\to\infty$, so \cref{eq:mean-resonance} yields $\mu_r\to C_{q,r_\ast}$, and hence $\E(X_{N,r})_t\to C_{q,r_\ast}^t$. Thus convergence of all fixed factorial moments to those of $\Pois(C_{q,r_\ast})$ gives the Poisson convergence.

It remains to locate the maximum null dimension. At dimension $r+1$, the exponential part of the first moment is $(r+1)(c-1)-m\binom{r+1}{k}=-(r+1)\left(1+\frac{(k-1)c}{r+1-k}\right)$, which tends to $-\infty$. Thus $\Pp(\alpha_N\ge r+1)\to0$ by the Gaussian-binomial upper bound and Markov's inequality.

For dimension $r-1$, if $k=r$ then every $(r-1)$-subspace is null. Otherwise $k\le r-1$, and $m\binom{r-1}{k}=c(r-k)\le(r-1)(c+1)$. The exponential part of its first moment is $(r-1)(c+1)-m\binom{r-1}{k}=c(k-1)+r-1\to\infty$. Moreover, for the parameters $r'=r-1$ and $c'=c+1$ one has $c'-r'=c-r+2\to\infty$. The pair-overlap estimate and its second-moment consequence therefore imply that a null $(r-1)$-subspace exists with probability tending to one. Consequently $\Pp(\alpha_N\in\{r-1,r\})\to1$. On this event, $\alpha_N=r$ if and only if $X_{N,r}>0$. The limiting probabilities now follow from the Poisson law.
\end{proof}

For every fixed $k,m$, the resonant subsequence is infinite. Indeed, if $r-1=kn$, then $(k-1)\binom{kn}{k-1}=kn\binom{kn-1}{k-2}$; since $\gcd(k,k-1)=1$, one has $k\mid\binom{r-1}{k-1}$. Thus $c_r=(m/k)\binom{r-1}{k-1}$ is integral for every $r\equiv1\pmod k$.

\section{Uniform quantitative Poisson approximation at fixed exterior order}\label{sec:fixed-target}
Throughout this section, unless stated otherwise, $k\ge3$ is fixed; whenever the random-model parameters are present, we work at resonance, uniformly over prime powers $q$ and integers $m\ge1$. Combining a high-excess entropy split with growing exterior stability up to normalized excess $\Gamma=O(r)$ yields factorial control through order $\rho_{m,k}(r)$.

At resonance,
\begin{equation}
 c=\frac mr\binom rk=\frac{m}{k!}r^{k-1}+O_k(mr^{k-2}),
 \qquad
 c\asymp_k mr^{k-1}.
 \label{eq:resonant-c-scale}
\end{equation}

\subsection{Growing exterior stability for every fixed order}
The low-excess stability theorem \cref{thm:stability-intro} can be made uniform throughout a linearly growing range of excess parameters. The proposition below makes this precise and provides the deterministic input that improves the higher-order Poisson rate. We refer to the range $1\le \Gamma\le \eta_k r$ appearing there as the \emph{growing-excess range}.

\begin{proposition}[Growing exterior stability]\label{prop:growing-exterior-stability}
There exist constants $\eta_k,C_k^{\mathrm{end}},A_k^{\mathrm{loc}}>0$, depending only on $k$, such that whenever $\Gamma$ lies in the growing-excess range, every ordering of a $\Gamma$-low-excess family $\mathcal H$ satisfies
\begin{equation}
 \sum_{j=2}^t\min\{d_j,r-d_j\}\le C_k^{\mathrm{end}}\Gamma.
 \label{eq:growing-endpoint-budget}
\end{equation}
With $D_\Gamma:=\lceil C_k^{\mathrm{end}}\Gamma\rceil$ and $r>2D_\Gamma$, call $H_j$ a \emph{new-block step} if $d_j\le D_\Gamma$ and an \emph{attachment step} if $r-d_j\le D_\Gamma$. Then every attachment step lies within Grassmann distance at most $A_k^{\mathrm{loc}}\Gamma$ of a unique earlier new-block space. Consequently, the canonical blocks of \cref{cor:low-excess-normal-form-intro} have top thickness $O_k(\Gamma)$, cross-block representative span deficiency $O_k(\Gamma)$, and pairwise Grassmann diameter $O_k(\Gamma)$, uniformly in $t$ and in the field.

Moreover, if $\kappa$ is the number of canonical blocks, there exist representatives $A_1,\ldots,A_\kappa$ and subspaces $A_i^\circ\le A_i$ such that $A_1^\circ+\cdots+A_\kappa^\circ$ is direct, $\sum_{i=1}^{\kappa}\operatorname{codim}_{A_i}A_i^\circ=O_k(\Gamma)$, and, with $\mathcal T=\bigoplus_{i=1}^{\kappa}\Lambda^kA_i^\circ$,
\[
\mathcal T\le\mathcal E,\qquad \dim(\mathcal E/\mathcal T)=O_k(\Gamma r^{k-1}).
\]
\end{proposition}

\begin{proof}
Recall $f_{r,k}(d)=\beta d-\binom dk$. By \cref{eq:uniform-endpoint-profile,eq:fixed-excess-sum}, summing over $j$ proves \cref{eq:growing-endpoint-budget} with the fixed choice $C_k^{\mathrm{end}}:=\tau_k^{-1}$. For $k=3$ the profile has the exact factorization
\begin{equation}
 f_{r,3}(d)=\beta d-\binom d3
 =\frac{d(r-d)(r+d-3)}6,
 \label{eq:trilinear-exact-profile}
\end{equation}
which gives the particularly transparent constant-order endpoint profile discussed earlier.

It remains to make the localization argument uniform throughout the growing-excess range. Fix an attachment step $H_j$, and let $A_1,\ldots,A_\ell$ be the preceding new-block spaces. Retain the notation $\mathcal O_{<j},\mathcal E_{<j},A_i^\circ,\mathcal T_0,L_j$ from the proof of \cref{lem:bounded-radius-clustering}. By \cref{eq:growing-endpoint-budget}, $|\dim\mathcal O_{<j}-\ell r|\le C_k^{\mathrm{end}}\Gamma$. Since the prefix excess is at most $\Gamma r^{k-1}$ and $\beta=O_k(r^{k-1})$,
\begin{equation}
 \left|\dim\mathcal E_{<j}-\ell\binom rk\right|
 \le C_{1,k}\Gamma r^{k-1}.
 \label{eq:growing-prefix-exterior}
\end{equation}
If $c_i=\operatorname{codim}_{A_i}A_i^\circ$, then $\sum_i c_i\le C_k^{\mathrm{end}}\Gamma$, and Pascal's identity gives $0\le \ell\binom rk-\dim\mathcal T_0\le \binom{r-1}{k-1}\sum_i c_i\le C_{2,k}\Gamma r^{k-1}$. Together with \cref{eq:growing-prefix-exterior}, this gives
\begin{equation}
 \dim(\mathcal E_{<j}/\mathcal T_0)\le C_{3,k}\Gamma r^{k-1}.
 \label{eq:growing-pure-quotient}
\end{equation}
The local excess is at most the total excess, hence $e_j\ge\beta d_j-\Gamma r^{k-1}$. Restricting the quotient map and using \cref{eq:growing-pure-quotient,lem:pure-block-intersection} yields
\begin{equation}
 \sum_{i=1}^{\ell}\binom{a_i}{k}
 =\dim(\Lambda^kL_j\cap\mathcal T_0)
 \ge \beta d_j-C^{(4)}_k\Gamma r^{k-1},
 \label{eq:growing-pure-overlap}
\end{equation}
where $a_i=\dim(L_j\cap A_i^\circ)$.

Choose the fixed $A_k^{\mathrm{loc}}>1$ sufficiently large and set $K_\Gamma=\lceil(A_k^{\mathrm{loc}}-1)\Gamma\rceil$. Then choose the fixed $\eta_k$ sufficiently small so that $D_\Gamma+2K_\Gamma<r/2$ throughout the stated range (in particular $D_\Gamma,K_\Gamma<r/3$). 

Suppose for contradiction that every $a_i\le r-K_\Gamma$. Since $\sum_i a_i\le r$, \cref{eq:capped-convex-concentration} with $L=K_\Gamma$ applies. Since $d_j\ge r-D_\Gamma$, it is enough to estimate the gap
\begin{align}
 \beta(r-D_\Gamma)-\binom{r-K_\Gamma}{k}-\binom{K_\Gamma}k
 &=\left[\binom rk-\binom{r-K_\Gamma}{k}\right]-\beta D_\Gamma-\binom{K_\Gamma}k.\label{eq:growing-localization-gap}
\end{align}
If $K_\Gamma\le r/3$, then
\[
 \binom rk-\binom{r-K_\Gamma}{k}
 =\sum_{h=0}^{K_\Gamma-1}\binom{r-1-h}{k-1}
 \ge c^{(5)}_k K_\Gamma r^{k-1}.
\]
Also $\beta D_\Gamma\le C^{(5)}_k\Gamma r^{k-1}$. Finally, by the fixed choice of $\eta_k$, for which $K_\Gamma/r$ is uniformly small,
\[
 \binom{K_\Gamma}k
 \le \frac{K_\Gamma^k}{k!}
 \le c^{(6)}_k\left(\frac{K_\Gamma}r\right)^{k-1}K_\Gamma r^{k-1}
 \le \frac{c^{(5)}_k}{4}K_\Gamma r^{k-1}.
\]
Thus, by the fixed choices of $A_k^{\mathrm{loc}}$ and $\eta_k$ above, \cref{eq:growing-localization-gap} exceeds $C^{(4)}_k\Gamma r^{k-1}$ throughout the growing-excess range. This contradicts \cref{eq:growing-pure-overlap}. Hence some $a_i\ge r-K_\Gamma$, and
\[
 d_{\mathrm{Gr}}(H_j,A_i)\le K_\Gamma\le A_k^{\mathrm{loc}}\Gamma.
\]
The localizing new-block space is unique. Indeed, suppose two earlier new-block spaces $A,B$ were both within distance $K_\Gamma$ of $H_j$. Then $d_{\mathrm{Gr}}(A,B)\le2K_\Gamma$. On the other hand, at the creation of the later one, its intersection with the span containing the earlier one had dimension at most $D_\Gamma$, so $d_{\mathrm{Gr}}(A,B)\ge r-D_\Gamma$, contradicting $D_\Gamma+2K_\Gamma<r$. The asserted canonical-block consequences now follow exactly as in the proof of \cref{cor:low-excess-normal-form-intro}, with every fixed constant replaced by $O_k(\Gamma)$.

Finally, order one representative from each canonical block first. Every representative after the first is a new-block step, so if $A_i^\circ$ is complementary in $A_i$ to $A_i\cap(A_1+\cdots+A_{i-1})$, then \cref{eq:growing-endpoint-budget} gives $\sum_i\operatorname{codim}_{A_i}A_i^\circ=O_k(\Gamma)$. Thus $\mathcal T=\bigoplus_i\Lambda^kA_i^\circ\le\mathcal E$ and
\[
0\le\kappa\binom rk-\dim\mathcal T
\le\binom{r-1}{k-1}\sum_i\operatorname{codim}_{A_i}A_i^\circ
=O_k(\Gamma r^{k-1}).
\]
Since $\mathfrak O=\kappa r+O_k(\Gamma)$ and $\mathfrak E=\beta\mathfrak O+G_k(\mathcal H)$, one also has $\mathfrak E\le\kappa\binom rk+O_k(\Gamma r^{k-1})$, which proves the quotient bound.
\end{proof}

Fix one choice of $\eta_k,C_k^{\mathrm{end}},A_k^{\mathrm{loc}}$ provided by \cref{prop:growing-exterior-stability} for the remainder of this section.

\begin{corollary}[Exterior-rank quantization]\label{cor:exterior-rank-quantization}
Under the hypotheses of \cref{prop:growing-exterior-stability}, if $\kappa$ is the number of canonical Grassmann blocks, then $\mathfrak O=\kappa r+O_k(\Gamma)$ and $\mathfrak E=\kappa\binom rk+O_k(\Gamma r^{k-1})$. In particular, if $\Gamma=o(r)$, then $\kappa=\operatorname{nint}\!\left(\mathfrak O/r\right)=\operatorname{nint}\!\left(\mathfrak E/\binom rk\right)$ for all sufficiently large $r$, where $\operatorname{nint}$ denotes nearest integer.
\end{corollary}

\begin{proof}
The two estimates were obtained in the proof of \cref{prop:growing-exterior-stability}. Dividing respectively by $r$ and $\binom rk\asymp_k r^k$ gives errors $O_k(\Gamma/r)=o(1)$, which proves the nearest-integer conclusion.
\end{proof}

\subsection{Growing low-excess configuration counts}
The stability theorem turns the geometric decomposition into an entropy bound: the main term is $q^{c\mathfrak O}$, corresponding to the ordinary dimension created by the tuple, while localization leaves only $q^{O_k(\Gamma t^2r)}$ additional freedom.

Throughout this subsection, $\Gamma$ lies in the growing-excess range.

\begin{proposition}[Uniform growing low-excess cluster count]\label{prop:growing-low-excess-cluster-count}
There exists $C_k^{\mathrm{cnt}}>0$, depending only on $k$, such that for every $t\ge1$ and every admissible $\mathfrak O$, the number of ordered distinct $\Gamma$-low-excess $t$-tuples $\mathcal H$ is at most
\begin{equation}
 q^{c\mathfrak O+C_k^{\mathrm{cnt}}\Gamma t^2r}.
 \label{eq:growing-low-excess-count}
\end{equation}
\end{proposition}

\begin{proof}
Use \cref{prop:growing-exterior-stability}; all endpoint defects and localization radii are $O_k(\Gamma)$. The first space has $q^{rc+O(1)}$ choices. Suppose the preceding span is $\mathcal O_{<j}$, with $\mathfrak O_{<j}=\dim \mathcal O_{<j}\le tr$.

At a new-block step with $d=d_j=O_k(\Gamma)$, the exact Grassmann intersection count gives
\[
 q^{(\mathfrak O_{<j}-d)(r-d)}\qbinom{\mathfrak O_{<j}}{d}\qbinom{N-\mathfrak O_{<j}}{r-d}.
\]
Using \cref{eq:qbinom-bounds}, its exponent is at most $c(r-d)+d(\mathfrak O_{<j}-d)\le c(r-d)+O_k(\Gamma tr)$.
At an attachment step put $s=r-d_j=O_k(\Gamma)$ and choose an earlier new-block space $A_0$ within distance $O_k(\Gamma)$. 

Write $B=H_j\cap A_0$, $L_j=H_j\cap\mathcal O_{<j}$, and $h=\operatorname{codim}_{A_0}B=O_k(\Gamma)$. Since $A_0\le\mathcal O_{<j}$, we have $B\le L_j$; as $\dim B=r-h$ and $\dim L_j=r-s$, necessarily $h\ge s$. 

For fixed $h$, there are at most $\qbinom rh\le C_q q^{h(r-h)}$ choices for $B$. Once $B$ is fixed, $L_j/B$ is an $(h-s)$-space in $\mathcal O_{<j}/B$, so there are at most $\qbinom{\mathfrak O_{<j}-r+h}{h-s}\le C_q q^{(h-s)(\mathfrak O_{<j}-r+s)}$ choices for $L_j$.

Now $Q=(H_j+\mathcal O_{<j})/\mathcal O_{<j}$ is an $s$-space in $V_N/\mathcal O_{<j}$, so there are at most $C_q q^{s(N-\mathfrak O_{<j}-s)}$ choices for $Q$. For fixed $L_j$ and $Q$, the remaining choice of $H_j$ is a graph map $Q\to\mathcal O_{<j}/L_j$, contributing at most $q^{s(\mathfrak O_{<j}-r+s)}$ possibilities. Thus the choices of $Q$ and the graph map contribute at most $C_q q^{s(N-r)}=C_q q^{sc}$. Since $h,s=O_k(\Gamma)$ and $\mathfrak O_{<j}\le tr$, summing over the $O_k(\Gamma)$ possible values of $h$ and absorbing $C_q\le C_2$ gives $q^{c(r-d_j)+O_k(\Gamma tr)}$ choices for every attachment step. Endpoint profiles, the new-block/attachment pattern, and choices of localization representatives contribute only $q^{O_k(\Gamma t^2r)}$ additional possibilities, since $C_q\le C_2$. Finally, $\mathfrak O=r+\sum_{j=2}^t(r-d_j)$. Multiplying over the steps and using the fixed $C_k^{\mathrm{cnt}}$ to dominate the implied constants proves \cref{eq:growing-low-excess-count}.
\end{proof}

\begin{corollary}[Explicit high-excess contribution]\label{cor:high-excess-explicit}
For every $t\ge1$ and $\Gamma>0$, the total contribution to $\E(X_{N,r})_t$ from ordered distinct $\Gamma$-high-excess tuples is at most
\begin{equation}
 trC_q^{t+1}
 q^{(t-1)^2r^2/4-m\Gamma r^{k-1}}.\label{eq:high-excess-explicit}
\end{equation}
\end{corollary}

\begin{proof}
Fix the ordinary span dimension $\mathfrak O$ and apply \cref{eq:span-tuple-count}. If $\mathcal H$ is $\Gamma$-high-excess, then $\mathfrak E>\beta \mathfrak O+\Gamma r^{k-1}$, so its joint null probability is at most $q^{-m\beta \mathfrak O-m\Gamma r^{k-1}}=q^{-c\mathfrak O-m\Gamma r^{k-1}}$.

Thus the contribution for fixed $\mathfrak O$ is at most $C_q^{t+1}q^{(\mathfrak O-r)(tr-\mathfrak O)-m\Gamma r^{k-1}}$. Since $(\mathfrak O-r)(tr-\mathfrak O)\le(t-1)^2r^2/4$ and there are at most $tr$ values of $\mathfrak O$, summing gives \cref{eq:high-excess-explicit}.
\end{proof}

\subsection{Growing factorial moments at the stability scale}

\begin{proposition}[Uniform growing factorial moments at fixed exterior order]\label{prop:growing-fixed-target-moments}
There exist constants $a_k^{\mathrm{mom}},b_k^{\mathrm{mom}}>0$, depending only on $k$, such that
\begin{equation}
 1\le t\le a_k^{\mathrm{mom}}\rho_{m,k}(r)
 \quad\Longrightarrow\quad
 \left|\E(X_r)_t-\mu_r^t\right|\le q^{-b_k^{\mathrm{mom}} m r^{k-1}}.
 \label{eq:growing-fixed-target-moment-error}
\end{equation}
Moreover, the entire non-direct contribution to $\E(X_r)_t$ is at most $q^{-b_k^{\mathrm{mom}} m r^{k-1}}$ uniformly in the same range.
\end{proposition}

\begin{proof}
Choose the fixed $a_k^{\mathrm{mom}}>0$ to satisfy all $k$-dependent smallness requirements below, and choose $b_k^{\mathrm{mom}}>0$ below all resulting positive exponential constants. The direct tuples differ from $\mu_r^t$ by $q^{-\Omega_k(mr^{k-1})}$ uniformly for $t\le a_k^{\mathrm{mom}}\rho_{m,k}(r)$. Indeed, by \cref{eq:resonant-c-scale}, $c\asymp_k mr^{k-1}$, whereas $tr=O_k(\sqrt m\,r^{k/2})=o_k(mr^{k-1})$ uniformly for $m\ge1$. \cref{eq:direct-tuple-deficit} gives the asserted direct-tuple error for the fixed choice of $a_k^{\mathrm{mom}}$; the generic non-direct term in that preliminary proposition is not used here.

It remains to bound non-direct tuples. We choose the excess cutoff so that the high-excess contribution in \cref{eq:high-excess-explicit} is exponentially small, while the remaining tuples lie in the growing-excess range where \cref{eq:growing-low-excess-count} applies. Put $\Gamma_*=2\left(1+\frac{t^2}{m r^{k-3}}\right)$. By the fixed choice of $a_k^{\mathrm{mom}}$, $2\Gamma_*\le\eta_k r$ throughout its stated range. For tuples with $G_k>\Gamma_*r^{k-1}$, \cref{eq:high-excess-explicit} gives total contribution at most
\[
 trC_q^{t+1}
 q^{(t-1)^2r^2/4-m\Gamma_*r^{k-1}}
 \le q^{-b_{1,k}m r^{k-1}}
\]
for some $b_{1,k}>0$. Since $C_q\le C_2$, the factor $C_q^ttr$ is $q^{o_k(mr^{k-1})}$ uniformly in $q$ and $m$.

For the remaining non-direct tuples, \cref{eq:strict-exterior-expansion} gives $G_k\ge\beta$. Since $k$ is fixed, $\beta/r^{k-1}\longrightarrow1/k!$, so every such tuple satisfies $G_k\ge \mathfrak S r^{k-1}$ for a constant $\mathfrak S>0$. Split $\mathfrak S r^{k-1}\le G_k\le\Gamma_*r^{k-1}$ into half-open dyadic strata $\gamma r^{k-1}\le G_k<2\gamma r^{k-1}$. For each ordinary span dimension $\mathfrak O$, apply \cref{prop:growing-low-excess-cluster-count} with $\widehat\gamma=\max\{1,2\gamma\}$. Since $\gamma\ge\mathfrak S$, one has $\widehat\gamma=O_k(\gamma)$, and because $\gamma\le\Gamma_*$ the growing stability range applies. The number of tuples in the stratum is therefore at most $q^{c\mathfrak O+C_k'\gamma t^2r}$, whereas each has joint null probability at most $q^{-c\mathfrak O-m\gamma r^{k-1}}$. The choice $t\le a_k^{\mathrm{mom}}\rho_{m,k}(r)$ gives $t^2r\le (a_k^{\mathrm{mom}})^2m r^{k-1}$, and the fixed $a_k^{\mathrm{mom}}$ is chosen small enough that
\[
 -m\gamma r^{k-1}+C_k'\gamma t^2r
 \le-\frac m2\gamma r^{k-1}
 \le-b_{2,k}m r^{k-1}.
\]
There are at most $tr$ possible values of $\mathfrak O$ and $O(\log r+\log(2+t))$ dyadic strata; these multiplicities are $q^{o_k(mr^{k-1})}$ uniformly in $m\ge1$. Hence all non-direct tuples contribute at most $q^{-b_k^{\mathrm{mom}}mr^{k-1}}$. Combining this with the direct-tuple estimate proves \cref{eq:growing-fixed-target-moment-error}.
\end{proof}

Fix these choices of $a_k^{\mathrm{mom}}$ and $b_k^{\mathrm{mom}}$ for the remainder of this section.

\subsection{Quantitative Brun inversion}

The following elementary Bonferroni inversion converts growing factorial-moment control into total variation.

\begin{lemma}[Quantitative Brun inversion]\label{lem:quantitative-brun}
Let $X$ be a finite sum of indicators, let $\lambda>0$, and let $T\ge1$ be an integer. If $\max_{1\le j\le T+1}\left|\E(X)_j-\lambda^j\right|\le\varepsilon$, then
\[
 d_{\mathrm{TV}}\!\left(\mathcal L(X),\Pois(\lambda)\right)
 \le e^2\varepsilon
 +\frac{(2^{T+2}+2)(\lambda^{T+1}+\varepsilon)}{(T+1)!}.
\]
\end{lemma}

\begin{proof}
Write $M_j=\E(X)_j$. The Bonferroni inequalities applied to the indicator family defining $X$ give, for $0\le j\le T$,
\[
 \left|\Pp(X=j)-\frac1{j!}
 \sum_{\ell=0}^{T-j}\frac{(-1)^\ell M_{j+\ell}}{\ell!}\right|
 \le\frac{M_{T+1}}{j!(T+1-j)!}.
\]
For a Poisson variable $Y\sim\Pois(\lambda)$, Taylor's theorem with remainder applied to $e^{-\lambda}$ gives analogously
\[
 \left|\Pp(Y=j)-\frac1{j!}
 \sum_{\ell=0}^{T-j}\frac{(-1)^\ell\lambda^{j+\ell}}{\ell!}\right|
 \le\frac{\lambda^{T+1}}{j!(T+1-j)!}.
\]
Summing the differences of the truncated series over $j\le T$ costs at most $e^2\varepsilon$, while
\[
 \sum_{j=0}^{T}\frac1{j!(T+1-j)!}
 \le\frac{2^{T+1}}{(T+1)!}.
\]
Finally, $\Pp(X>T)\le M_{T+1}/(T+1)!$, and the same factorial-Markov bound applies to the Poisson tail. Since $M_{T+1}\le\lambda^{T+1}+\varepsilon$, the displayed estimate follows.
\end{proof}

\subsection{From counts to the whole null configuration}\label{subsec:process-upgrade}

Retain the notation $\mathcal G_r,p_r,\mathcal B_r$ from \cref{thm:fixed-target-intro}. A subset $\mathcal S\subseteq\mathcal G_r$ is called \emph{direct} if $\dim\sum_{H\in\mathcal S}H=r|\mathcal S|$.

Let $\mathfrak D_r$ denote the collection of direct subsets of $\mathcal G_r$. The next elementary symmetry lemma is the mechanism that upgrades a scalar count approximation to an approximation of the entire random set.

\begin{lemma}[Orbit reduction]\label{lem:orbit-reduction}
Let a finite group $G$ act on a finite set $\Omega$, and let $\mathfrak D\subseteq2^\Omega$ be $G$-invariant. Assume that, for every $t$, the action of $G$ is transitive on $\mathfrak D_t:=\{S\in\mathfrak D:|S|=t\}$ whenever this set is nonempty. If $\mathcal U,\mathcal V\subseteq\Omega$ are random subsets whose laws are $G$-invariant, then
\begin{equation}
 d_{\mathrm{TV}}\!\left(\mathcal L(\mathcal U),\mathcal L(\mathcal V)\right)
 \le
 d_{\mathrm{TV}}\!\left(\mathcal L(|\mathcal U|),\mathcal L(|\mathcal V|)\right)
 +\Pp(\mathcal U\notin\mathfrak D)+\Pp(\mathcal V\notin\mathfrak D).
 \label{eq:orbit-reduction}
\end{equation}
\end{lemma}

\begin{proof}
Put $a_t=\Pp(\mathcal U\in\mathfrak D_t)$ and $b_t=\Pp(\mathcal V\in\mathfrak D_t)$. By invariance and transitivity, conditional on belonging to $\mathfrak D_t$ both random sets are uniform on $\mathfrak D_t$. Hence the contribution of $\mathfrak D$ to twice the total-variation distance is $\sum_t|a_t-b_t|$, while the contribution of its complement is at most $\Pp(\mathcal U\notin\mathfrak D)+\Pp(\mathcal V\notin\mathfrak D)$. If $u_t=\Pp(|\mathcal U|=t),\quad v_t=\Pp(|\mathcal V|=t)$, and $u_t^{\mathrm{out}},v_t^{\mathrm{out}}$ are the corresponding probabilities of being outside $\mathfrak D$, then $a_t=u_t-u_t^{\mathrm{out}}$ and $b_t=v_t-v_t^{\mathrm{out}}$.

Therefore $\sum_t|a_t-b_t|\le 2d_{\mathrm{TV}}\!\left(\mathcal L(|\mathcal U|),\mathcal L(|\mathcal V|)\right)+\Pp(\mathcal U\notin\mathfrak D)+\Pp(\mathcal V\notin\mathfrak D)$, which gives \cref{eq:orbit-reduction}.
\end{proof}

\begin{proposition}[Uniform directness at resonance]\label{prop:uniform-direct-fixed}
\begin{equation}
\Pp(\mathcal Z_{N,r}\notin\mathfrak D_r)+\Pp(\mathcal B_r\notin\mathfrak D_r)
\le\exp\!\left(-4a_k^{\mathrm{TV}}\Psi_r\right).
\label{eq:uniform-direct-fixed}
\end{equation}
\end{proposition}

\begin{proof}
For $t\ge2$, let $R_t$ be the number of ordered distinct non-direct $t$-tuples contained in $\mathcal Z_{N,r}$. Thus $\E R_t$ is exactly the non-direct contribution to $\E(X_r)_t$.

Choose $T=\left\lfloor \theta_k^{\mathrm{dir}}\rho_{m,k}(r)\right\rfloor$ with a fixed $0<\theta_k^{\mathrm{dir}}<a_k^{\mathrm{mom}}/2$. Then uniformly for $2\le t\le T+1$,
\[
 \E R_t\le q^{-b_k^{\mathrm{mom}}mr^{k-1}},
 \qquad
 \left|\E(X_r)_t-\mu_r^t\right|\le q^{-b_k^{\mathrm{mom}}mr^{k-1}}.
\]

Choose the fixed $a_k^{\mathrm{TV}}>0$ sufficiently small so that the Stirling bound below yields the estimate $\exp\!\left(-4a_k^{\mathrm{TV}}T\log T\right)$, and so that all later absorptions in the proof of \cref{thm:fixed-target-intro} remain valid.

If $\mathcal Z_{N,r}$ is non-direct and $X_r=s\le T$, then every ordering of its $s$ members is a non-direct $s$-tuple, so $R_s=s!$. Therefore
\begin{align*}
 \Pp(\mathcal Z_{N,r}\notin\mathfrak D_r)
 &\le \Pp(X_r>T)+\sum_{s=2}^{T}\frac{\E R_s}{s!}\\
 &\le \frac{\E(X_r)_{T+1}}{(T+1)!}+e q^{-b_k^{\mathrm{mom}}mr^{k-1}}\\
 &\le \frac{\mu_r^{T+1}+q^{-b_k^{\mathrm{mom}}mr^{k-1}}}{(T+1)!}+e q^{-b_k^{\mathrm{mom}}mr^{k-1}}.
\end{align*}
Since $\mu_r\le C_2$, Stirling's formula and the fixed choice of $a_k^{\mathrm{TV}}$ bound this by $\exp(-4a_k^{\mathrm{TV}}T\log T)$.

Conditional on $|\mathcal B_r|=s$, the set $\mathcal B_r$ is uniform among the $s$-subsets of $\mathcal G_r$. Let $D_s$ be the number of ordered direct $s$-tuples. By \cref{eq:direct-tuple-deficit}, uniformly for $s\le T$,
\[
 |\mathcal G_r|^s-D_s\le 4s q^{-c+(s-1)r}|\mathcal G_r|^s.
\]
Also $(|\mathcal G_r|)_s/|\mathcal G_r|^s=1-o(1)$ uniformly for $s\le T$: indeed $T^2=O_k(mr^{k-2})$, whereas \cref{eq:resonant-c-scale} gives $rc\asymp_k mr^k$ and $c\asymp_k mr^{k-1}$. Thus $|\mathcal G_r|\ge q^{rc}$ is exponentially larger than $T^2$. Consequently
\[
 \Pp(\mathcal B_r\notin\mathfrak D_r\mid |\mathcal B_r|=s)
 \le 8s q^{-c+(s-1)r}
 =q^{-\Omega_k(mr^{k-1})}
\]
uniformly for $s\le T$, since $Tr=o_k(mr^{k-1})$ uniformly in $m\ge1$. Finally,
\[
 \Pp(|\mathcal B_r|>T)
 \le\frac{\E\!\left[(|\mathcal B_r|)_{T+1}\right]}{(T+1)!}
 =\frac{(|\mathcal G_r|)_{T+1}p_r^{T+1}}{(T+1)!}
 \le\frac{\mu_r^{T+1}}{(T+1)!}.
\]
Combining these bounds and using $T\asymp_k\rho_{m,k}(r)$ proves \cref{eq:uniform-direct-fixed}.
\end{proof}

\begin{lemma}[Bernoulli to Poisson point process]\label{lem:bernoulli-poisson-process}
Let $\Omega$ be finite, let $\mathcal B\subseteq\Omega$ retain each point independently with probability $p$, and put $\Xi_{\mathcal B}=\sum_{x\in\mathcal B}\delta_x$. If $\Pi$ is the Poisson point process on $\Omega$ with intensity $p\sum_{x\in\Omega}\delta_x$, then $d_{\mathrm{TV}}\!\left(\mathcal L(\Xi_{\mathcal B}),\mathcal L(\Pi)\right)\le|\Omega|p^2$. The same bound holds between $\operatorname{Bin}(|\Omega|,p)$ and $\Pois(|\Omega|p)$.
\end{lemma}

\begin{proof}
At one site, $d_{\mathrm{TV}}\!\left(\operatorname{Ber}(p),\Pois(p)\right)=p(1-e^{-p})\le p^2$. For product measures, total variation satisfies $d_{\mathrm{TV}}\!\left(\bigotimes_{x\in\Omega}P_x,\bigotimes_{x\in\Omega}Q_x\right)\le\sum_{x\in\Omega}d_{\mathrm{TV}}(P_x,Q_x)$. Applying this with $P_x=\operatorname{Ber}(p)$ and $Q_x=\Pois(p)$ gives $d_{\mathrm{TV}}\!\left(\mathcal L(\Xi_{\mathcal B}),\mathcal L(\Pi)\right)\le|\Omega|p^2$. Applying the total-mass map gives the binomial--Poisson bound.
\end{proof}

\subsection{Proof of the uniform critical random-set theorem}
\label{subsec:proof-uniform-critical-random-set}

\begin{proof}[Proof of \cref{thm:fixed-target-intro}]
We first prove the scalar approximation. Set $T=\left\lfloor\frac{a_k^{\mathrm{mom}}}{2}\rho_{m,k}(r)\right\rfloor$. Then, uniformly for $1\le t\le T+1,\quad\left|\E(X_r)_t-\mu_r^t\right|\le q^{-b_k^{\mathrm{mom}}mr^{k-1}}$.

The condition $tr\le N$ required in the preliminary direct-tuple estimates holds for every $t\le T+1$, since $(T+1)r=O_k(\sqrt m\,r^{k/2})=o_k(N)$ uniformly for $m\ge1$, by \cref{eq:resonant-c-scale}. Since $\mu_r\le C_q\le C_2$, applying \cref{lem:quantitative-brun} to $X_r$ with Poisson mean $\mu_r$, together with Stirling's formula, gives
\[
\frac{(2C_2)^{T+1}}{(T+1)!}=\exp(-T\log T+O(T)).
\]
Because $T\asymp_k\rho_{m,k}(r)$, the fixed choice of $a_k^{\mathrm{TV}}$ gives
\begin{equation}
d_{\mathrm{TV}}\!\left(\mathcal L(X_r),\Pois(\mu_r)\right)
\le \exp\!\left(-4a_k^{\mathrm{TV}}\Psi_r)\right)
= e^{-4a_k^{\mathrm{TV}}\Psi_r}.
\label{eq:critical-scalar-tv-rate}
\end{equation}
In particular $d_{\mathrm{TV}}\!\left(\mathcal L(X_r),\Pois(\mu_r)\right)\le e^{-a_k^{\mathrm{TV}}\Psi_r}$ proving \cref{eq:critical-count-poisson}.

We now upgrade the count law to the entire random subset. Since $|\mathcal B_r|\sim\operatorname{Bin}(|\mathcal G_r|,p_r)$ and $|\mathcal G_r|p_r=\mu_r$, \cref{lem:bernoulli-poisson-process} gives $d_{\mathrm{TV}}\!\left(\mathcal L(|\mathcal B_r|),\Pois(\mu_r)\right)\le |\mathcal G_r|p_r^2=\mu_rp_r$. Hence, by \cref{eq:critical-scalar-tv-rate},
\begin{equation}
d_{\mathrm{TV}}\!\left(\mathcal L(X_r),\mathcal L(|\mathcal B_r|)\right)
\le e^{-4a_k^{\mathrm{TV}}\Psi_r}+\mu_rp_r.
\label{eq:count-to-bernoulli-count}
\end{equation}

The group $\mathrm{GL}(V_N)$ acts transitively on the direct $t$-subsets of $\mathcal G_r$ whenever this collection is nonempty. Indeed, given two such subsets, choose a bijection between their members, choose bases in the corresponding summands, and extend the resulting independent sets to bases of $V_N$; the induced linear automorphism carries one direct $t$-subset to the other. The law of $\mathcal B_r$ is plainly invariant under this action. The law of $\mathcal Z_{N,r}$ is also invariant: for $g\in\mathrm{GL}(V_N)$, the map $\Theta_N\circ\Lambda^k g^{-1}$ is again uniform in $\operatorname{Hom}(\Lambda^kV_N,\F_q^m)$, and its null configuration is $g\mathcal Z_{N,r}$.

Applying \cref{lem:orbit-reduction} with $\mathfrak D_r$, then using \cref{prop:uniform-direct-fixed,eq:count-to-bernoulli-count}, gives
\[
d_{\mathrm{TV}}\!\left(\mathcal L(\mathcal Z_{N,r}),\mathcal L(\mathcal B_r)\right)
\le 2e^{-4a_k^{\mathrm{TV}}\Psi_r}+\mu_rp_r.
\]
Under the identification of a subset with its simple counting measure, \cref{lem:bernoulli-poisson-process} and the triangle inequality likewise give
\[
d_{\mathrm{TV}}\!\left(\mathcal L(\Xi_r),\mathcal L(\Pi_r)\right)
\le 2e^{-4a_k^{\mathrm{TV}}\Psi_r}+2\mu_rp_r.
\]
Since $p_r=q^{-rc}$, $\mu_r\le C_2$, and $rc\gg_k\Psi_r$, both right-hand sides are at most $e^{-a_k^{\mathrm{TV}}\Psi_r}$ for all sufficiently large $r$, uniformly in $q$ and $m$. This proves \cref{eq:bernoulli-subset-tv,eq:poisson-process-tv}.

Finally, \cref{eq:mean-resonance} and $C_q\le C_2$ give, for an absolute constant $B>0$, $|\mu_r-C_q|\le Bq^{-r}$, uniformly for all sufficiently large $r$; here $c\ge r$ since $k\ge3$ and $c\asymp_k mr^{k-1}$ by \cref{eq:resonant-c-scale}. The elementary bound $d_{\mathrm{TV}}\!\left(\Pois(\mu_r),\Pois(C_q)\right)\le |\mu_r-C_q|$ and the triangle inequality with \cref{eq:critical-count-poisson} therefore give
\[
d_{\mathrm{TV}}\!\left(\mathcal L(X_{N,r}),\Pois(C_q)\right)
\le Bq^{-r}+e^{-a_k^{\mathrm{TV}}\Psi_r},
\]
which proves \cref{eq:uniform-count-to-Cq}.

Along every resonant sequence $q=q_r$ and $m=m_r\ge1$, we have $q_r^{-r}\to0$ and $\Psi_r\longrightarrow\infty$, so \cref{eq:uniform-count-to-Cq} yields $d_{\mathrm{TV}}\!\left(\mathcal L(X_{N,r}),\Pois(C_{q_r})\right)\longrightarrow0$. If moreover $q_r\to\infty$, then $C_{q_r}\to1$, and hence $X_{N,r}\xrightarrow d\Pois(1)$.
\end{proof}

\subsection{Rare null \texorpdfstring{$(r+1)$}{(r+1)}-spaces and a high-moment obstruction}
The proof above controls factorial moments through order $\rho_{m,k}(r)$. At larger factorial orders, rare null $(r+1)$-spaces provide a distinct obstruction.

The mechanism can be summarized directly at the level of expectations. Every null $(r+1)$-space contains $L_r=\qbinom{r+1}{r}\asymp_q q^r$ null $r$-dimensional hyperplanes. Moreover, for $t\ge2$, two distinct hyperplanes determine their ambient $(r+1)$-space. Consequently, the ordered $t$-tuples arising from different null $(r+1)$-spaces do not overlap, and $\E(X_r)_t\ge (L_r)_t\,\E X_{N,r+1}$. Thus an exponentially small expected number of larger null spaces can nevertheless produce a large contribution to sufficiently high factorial moments; \cref{prop:rplusone-space-moment-obstruction} quantifies this effect.

Before doing so, we record a converse rigidity statement: if an $(r+1)$-space contains too many null hyperplanes, then the whole space is null; see \cref{lem:null-hyperplane-saturation}.

\begin{lemma}[Hyperplane saturation gap in an $(r+1)$-space]
\label{lem:null-hyperplane-saturation}
Let $W$ be an $(r+1)$-dimensional vector space over $\F_q$, let $2\le k\le r$, and let $\Theta:\Lambda^kW\to\F_q^m$ be linear. Then exactly one of the following holds:
\begin{enumerate}[label=\textup{(\roman*)}]
\item $\Theta=0$, in which case every $r$-dimensional hyperplane of $W$ is null;
\item $\Theta\ne0$, in which case at most $(q^k-1)/(q-1)$ $r$-dimensional hyperplanes $H\le W$ satisfy $\Theta(\Lambda^kH)=0$.
\end{enumerate}
In particular, if more than $(q^k-1)/(q-1)$ hyperplanes of $W$ are null, then $W$ is null.
\end{lemma}

\begin{proof}
The first case is immediate. Suppose $\Theta\ne0$. Choose $\lambda\in(\F_q^m)^*$ such that $\omega:=\lambda\circ\Theta\in\Lambda^kW^*$ is nonzero. Every hyperplane null for $\Theta$ is also null for $\omega$.

Write an $r$-dimensional hyperplane as $H=\ker\phi$ for some $0\ne\phi\in W^*$, defined up to a nonzero scalar. Choose a decomposition $W^*=\langle\phi\rangle\oplus U$. Then every $\omega\in\Lambda^kW^*$ has a unique decomposition $\omega=\phi\wedge\eta+\xi$ where $\eta\in\Lambda^{k-1}U$ and $\xi\in\Lambda^kU$. The restriction of $\omega$ to $H$ is precisely $\xi$. Hence
\[
\omega|_{\Lambda^kH}=0
\quad\Longleftrightarrow\quad
\omega\in\phi\wedge\Lambda^{k-1}W^*.
\]

Let $D(\omega):=\{0\}\cup\{\phi\in W^*\setminus\{0\}:\omega\in\phi\wedge\Lambda^{k-1}W^*\}$. We claim that $D(\omega)$ is a vector subspace of $W^*$ of dimension at most $k$. Indeed, let $\phi_1,\ldots,\phi_d$ be a maximal linearly independent family in $D(\omega)$. Expanding $\omega$ in a basis beginning with $\phi_1,\ldots,\phi_d$, divisibility by each $\phi_i$ forces every nonzero basis monomial of $\omega$ to contain every $\phi_i$. Thus
\[
\omega\in\phi_1\wedge\cdots\wedge\phi_d\wedge\Lambda^{k-d}W^*.
\]
Since $\omega\ne0$, necessarily $d\le k$. Moreover every nonzero element of $\Span(\phi_1,\ldots,\phi_d)$ divides the displayed wedge, so it belongs to $D(\omega)$; maximality gives $D(\omega)=\Span(\phi_1,\ldots,\phi_d)$. Hyperplanes null for $\omega$ correspond to one-dimensional subspaces of $D(\omega)$. Hence the number of hyperplanes null for $\Theta$ is at most $\qbinom{k}{1}=(q^k-1)/(q-1)$.
\end{proof}

We now quantify the contribution of null $(r+1)$-spaces to the factorial moments of $X_r$.

\begin{proposition}[Null $(r+1)$-space moment obstruction]\label{prop:rplusone-space-moment-obstruction}
Uniformly over every prime power $q$, consider any resonant parameters $2\le k\le r$, $m\ge1$. Let $X_{N,r+1}$ count null $(r+1)$-spaces and put $L_r=\qbinom{r+1}{r}=(q^{r+1}-1)/(q-1)$. Then
\begin{equation}
 \E X_{N,r+1}
 \le C_q q^{-A_r},
 \qquad
 A_r=(r+1)\left(1+\frac{(k-1)c}{r+1-k}\right).
 \label{eq:rplusone-space-probability}
\end{equation}
Moreover, for every integer $2\le t\le L_r/2$,
\begin{equation}
 \log_q\E(X_r)_t
 \ge tr-A_r-O(t).
 \label{eq:rplusone-space-moment-lower}
\end{equation}
Consequently:
\begin{enumerate}[label=\textup{(\roman*)}]
 \item if $k=2$ and $m$ is fixed, every fixed integer $t>(m+2)/2$ has $\E(X_r)_t\to\infty$;
 \item fix $k\ge3$ and allow $m=m(r)\ge1$ to vary. For every $\varepsilon>0$, if
 \[
  t\ge\left(\frac{k-1}{k!}+\varepsilon\right)m r^{k-2},
  \qquad t\le L_r/2,
 \]
then, uniformly in $q$, $\E(X_r)_t/\mu_r^t\ge q^{(\varepsilon/2)mr^{k-1}}\longrightarrow\infty$.
\end{enumerate}
Thus rare null $(r+1)$-spaces are a universal obstruction to Poisson factorial-moment approximation at growing orders, even in regimes where the entire null configuration is approximated in total variation by the independent Bernoulli model.
\end{proposition}

\begin{proof}
By the Gaussian-binomial upper bound, $\E X_{N,r+1}\le C_q q^{(r+1)(N-r-1)-m\binom{r+1}{k}}$. At resonance, using $\binom{r+1}{k}=\binom rk\frac{r+1}{r+1-k}$, the exponent equals $(r+1)(c-1)-rc\frac{r+1}{r+1-k}=-A_r$, which proves \cref{eq:rplusone-space-probability}.

For the moment lower bound, every null $(r+1)$-space $W$ contains exactly $L_r$ $r$-dimensional hyperplanes, and every ordered $t$-tuple of distinct such hyperplanes is a null tuple. Any two distinct hyperplanes determine $W$, so these tuples are counted without overlap between different $(r+1)$-spaces. Hence
\begin{equation}
 \E(X_r)_t
 \ge \qbinom N{r+1}(L_r)_t q^{-m\binom{r+1}{k}}.
 \label{eq:rplusone-space-tuple-lower}
\end{equation}
For $t\le L_r/2$, $(L_r)_t\ge(L_r/2)^t$, while $\log_q L_r=r+O(1)$ uniformly in $q$ and the lower Gaussian-binomial bound gives $\qbinom N{r+1}\ge q^{(r+1)(N-r-1)}$. Taking logarithms in \cref{eq:rplusone-space-tuple-lower} yields \cref{eq:rplusone-space-moment-lower}.

If $k=2$ and $c=m(r-1)/2$, then $A_r/r\to(m+2)/2$, proving part \textup{(i)}. For part \textup{(ii)}, with $k\ge3$ fixed, \cref{eq:resonant-c-scale} and the definition of $A_r$ give $A_r=\frac{m(k-1)}{k!}r^{k-1}+O_k(mr^{k-2})$.

Since $1\le\mu_r\le C_q\le C_2$, we have $\log_q\mu_r^t=O(t)$, so
\[
\log_q\frac{\E(X_r)_t}{\mu_r^t}
\ge tr-A_r-O(t)
\ge(\varepsilon/2)mr^{k-1},
\]
which proves part~\textup{(ii)}.
\end{proof}

\begin{remark}[Two factorial scales]\label{rem:two-factorial-scales}
The uniform random-set proof controls factorial moments through order $\rho_{m,k}(r)$. Whenever $m r^{k-2}=o(q^r)$, in particular when $m$ is fixed, the condition $t\le L_r/2$ in \cref{prop:rplusone-space-moment-obstruction} is automatic at orders $t\asymp m r^{k-2}$, and the null $(r+1)$-space construction forces moment divergence at order
\[
m r^{k-2}=\rho_{m,k}(r)^2.
\]
Thus in these regimes the process law is already asymptotically Poisson on the first scale even though raw factorial moments must eventually fail on its square. For unrestricted growth of $m$, the exact statement is \cref{prop:rplusone-space-moment-obstruction}, including its stated range $t\le L_r/2$.
\end{remark}

\begin{proposition}[First dependence obstruction to total variation]
\label{prop:tv-lower-obstruction}
Fix $k\ge3$, and put $\chi_k:=(k-1)/k!$. Along every resonant sequence $r\to\infty$, allowing the prime power $q=q_r$ and the target dimension $m=m_r\ge1$ to vary,
\begin{equation}
 d_{\mathrm{TV}}\!\left(
   \mathcal L(X_{N,r}),\Pois(\mu_{N,r})
 \right)
 \ge
 q^{-(\chi_k+o_k(1))m r^{k-1}}.
 \label{eq:tv-lower-obstruction-count}
\end{equation}
The same lower bound holds for $d_{\mathrm{TV}}\!\left(\mathcal L(\mathcal Z_{N,r}),\mathcal L(\mathcal B_r)\right)$ and $d_{\mathrm{TV}}\!\left(\mathcal L(\Xi_r),\mathcal L(\Pi_r)\right)$. Here the $o_k(1)$ tends to zero with $r$, uniformly in $q$ and $m$. Equivalently,
\[
\limsup_{r\to\infty}
\frac{-\log_q d_{\mathrm{TV}}\!\left(\mathcal L(X_{N,r}),\Pois(\mu_{N,r})\right)}
{m r^{k-1}}
\le\frac{k-1}{k!},
\]
and likewise for the random-set and point-process distances.
\end{proposition}

\begin{proof}
Recall that $p_r=q^{-m\binom rk}=q^{-rc}$ and $\mu_r=|\mathcal G_r|p_r$. For a uniformly chosen $r$-space $H_2$, with $H_1$ fixed, let $\pi_{d_2}=\Pp(\dim(H_1\cap H_2)=d_2)$. By \cref{eq:exact-pair-factorial-moment} and $\sum_{d_2=0}^r\pi_{d_2}=1$,
\begin{equation}
 \frac{\E(X_r)_2-\mu_r^2}{\mu_r^2}
 =
 \sum_{d_2=1}^{r-1}\pi_{d_2}
 \left(q^{m\binom {d_2}k}-1\right)-\pi_r.
 \label{eq:second-moment-positive-decomposition}
\end{equation}
Taking the $d_2=r-1$ term in \cref{eq:second-moment-positive-decomposition}, we estimate its contribution explicitly. By \cref{eq:grassmann-intersection-count}, $\pi_{r-1}=\qbinom r1\qbinom c1 q/\qbinom Nr$. Using \cref{eq:qbinom-bounds} and $C_q\le C_2$, $\pi_{r-1}\ge C_2^{-1}q^{r+c-1-rc}$. At resonance, $m\binom{r-1}{k-1}=kc$, and hence $m\binom{r-1}{k}=m\binom rk-m\binom{r-1}{k-1}=(r-k)c$. Therefore, for all sufficiently large $r$,
\[
\pi_{r-1}\left(q^{m\binom{r-1}{k}}-1\right)
\ge \frac1{2C_2}q^{-(k-1)c+r-1}.
\]
On the other hand, $\pi_r=\qbinom Nr^{-1}\le q^{-rc}$, which is exponentially smaller. Since $\mu_r\ge1$, if $\zeta_r:=\E(X_r)_2-\mu_r^2$, then
\begin{equation}
\zeta_r
\ge q^{-(k-1)c+r-O(1)}
=q^{-(\chi_k+o_k(1))m r^{k-1}}.
\label{eq:second-moment-distance-one-gap}
\end{equation}

It remains to convert this second-moment discrepancy, which involves the unbounded statistic $(X_r)_2$, into a total-variation lower bound. Put $s=\left\lfloor \frac{a_k^{\mathrm{mom}}}{2}\rho_{m,k}(r)\right\rfloor$. Then $s\to\infty$, and \cref{prop:growing-fixed-target-moments} gives
\[
\E(X_r)_s\le\mu_r^s+q^{-b_k^{\mathrm{mom}}m r^{k-1}}\le C_2^s+1.
\]
Set $U_r=\left\lceil q^{\,2\chi_km r^{k-1}/(s-2)}\right\rceil$. For large $r$, $U_r\ge2s$. If $n>U_r$, then
\[
(n)_s=(n)_2(n-2)\cdots(n-s+1)\ge(n)_2\left(\frac{U_r}{2}\right)^{s-2}.
\]
Consequently, $\E\!\left[(X_r)_2;X_r>U_r\right]\le(C_2^s+1)\left(\frac2{U_r}\right)^{s-2}=q^{-(2\chi_k-o_k(1))m r^{k-1}}=o(\zeta_r)$.

Let $P_r\sim\Pois(\mu_r)$ and define $F_r(n):=(n)_2\mathbf 1_{\{n\le U_r\}}$. Then $0\le F_r\le U_r^2$, while
\begin{align*}
\E F_r(X_r)-\E F_r(P_r)
&=\zeta_r-\E[(X_r)_2;X_r>U_r]+\E[(P_r)_2;P_r>U_r]\\
&\ge\zeta_r-o(\zeta_r).
\end{align*}
Hence, by the dual characterization of total variation,
\[
d_{\mathrm{TV}}\!\left(\mathcal L(X_r),\Pois(\mu_r)\right)
\ge\frac{\zeta_r-o(\zeta_r)}{U_r^2}.
\]
Since $\log_q U_r^2/(m r^{k-1})=O_k(1/s)=o_k(1)$, \cref{eq:second-moment-distance-one-gap} proves \cref{eq:tv-lower-obstruction-count}.

For the random-set statement, applying the cardinality map gives
\[
d_{\mathrm{TV}}\!\left(\mathcal L(\mathcal Z_{N,r}),\mathcal L(\mathcal B_r)\right)
\ge d_{\mathrm{TV}}\!\left(\mathcal L(X_r),\mathcal L(|\mathcal B_r|)\right).
\]
Moreover, by \cref{lem:bernoulli-poisson-process},
\[
d_{\mathrm{TV}}\!\left(\mathcal L(|\mathcal B_r|),\Pois(\mu_r)\right)
\le\mu_rp_r\le C_2q^{-rc}.
\]
Since $rc=\Theta_k(mr^k)$, this is negligible compared with \cref{eq:tv-lower-obstruction-count}. The triangle inequality therefore gives the same lower bound for the random-set distance.

Finally, total mass maps $\Xi_r$ to $X_r$ and $\Pi_r$ to $\Pois(\mu_r)$, so total variation cannot increase under this map. Hence the same lower bound holds for the point-process distance.
\end{proof}

\begin{corollary}[Occupancy independence and direct-sum position]\label{cor:occupancy-independence}
Under the hypotheses of \cref{thm:fixed-target-intro}, the following hold uniformly in $q$ and $m$.
\begin{enumerate}[label=\textup{(\roman*)}]
 \item With probability at least $1-\exp\!\left(-a_k^{\mathrm{TV}}\Psi_r\right)$, the null $r$-spaces are in ordinary direct-sum position; equivalently, $\dim\sum_{H\in\mathcal Z_{N,r}}H=rX_{N,r}$.
 \item For every finite family of pairwise disjoint deterministic sets
 $\mathcal A_{r,1},\ldots,\mathcal A_{r,s}\subseteq\mathcal G_r$ (with $s$ allowed to depend on $r$),
 \[
 d_{\mathrm{TV}}\!\left(
 \mathcal L\left(\bigl(|\mathcal Z_{N,r}\cap\mathcal A_{r,i}|\bigr)_{i=1}^s\right),
 \bigotimes_{i=1}^s\operatorname{Bin}(|\mathcal A_{r,i}|,p_r)
 \right)
 \le \exp\!\left(-a_k^{\mathrm{TV}}\Psi_r\right).
 \]
 In particular, if $s$ is fixed and $p_r|\mathcal A_{r,i}|\to\lambda_i<\infty$, then these counts converge jointly to independent $\Pois(\lambda_i)$ variables.
 \item For every $t$ for which the conditioning event is nonempty, one has the exact symmetry statement
 \[
  \mathcal L\!\left(\mathcal Z_{N,r}\mid X_{N,r}=t,\ \mathcal Z_{N,r}\in\mathfrak D_r\right)
  =\operatorname{Unif}(\mathfrak D_{r,t}),
 \]
 where $\mathfrak D_{r,t}$ is the set of direct $t$-subsets of $\mathcal G_r$. Consequently, for every fixed $t\ge0$, if $\mathcal U_{r,t}$ is a uniformly random $t$-subset of $\mathcal G_r$, then
 \[
 d_{\mathrm{TV}}\!\left(
 \mathcal L(\mathcal Z_{N,r}\mid X_{N,r}=t),
 \mathcal L(\mathcal U_{r,t})
 \right)\longrightarrow0.
 \]
Moreover, conditional on $X_{N,r}=t$ and directness, the ordinary span $\sum_{H\in\mathcal Z_{N,r}}H$ is exactly uniform in $\Gr(tr,V_N)$. For fixed $t$, conditional only on $X_{N,r}=t$, the only exception to this uniform description comes from non-direct configurations, whose conditional probability is at most $q^{-\Omega_{k,t}(m r^{k-1})}$, uniformly in $q$ and $m$.
\end{enumerate}
\end{corollary}

\begin{proof}
Part \textup{(i)} is \cref{prop:uniform-direct-fixed}. Part \textup{(ii)} follows from \cref{eq:bernoulli-subset-tv} by applying the measurable map that records the indicated occupancy counts; under $\mathcal B_r$ those counts are independent binomials. The Poisson conclusion follows from the elementary binomial--Poisson bound in \cref{lem:bernoulli-poisson-process}.

For \textup{(iii)}, conditional on $X_{N,r}=t$ and on directness, $\mathrm{GL}(V_N)$-invariance makes $\mathcal Z_{N,r}$ exactly uniform over the direct $t$-subsets. The uniform $t$-subset $\mathcal U_{r,t}$ is direct with probability $1-q^{-\Omega_{k,t}(mr^{k-1})}$ by the direct-tuple estimate, uniformly in $q$ and $m$. The non-direct contribution $\E R_t$ has the same form of bound by \cref{prop:growing-fixed-target-moments}. Since $1\le C_q\le C_2$ and \cref{eq:critical-scalar-tv-rate} approximates $X_{N,r}$ uniformly by $\Pois(\mu_r)$ with $\mu_r=C_q+o(1)$, for each fixed $t$ the probability $\Pp(X_{N,r}=t)$ is bounded below by a positive constant depending only on $t$ uniformly in $q$ and $m$. Thus the conditional probability of non-directness has the claimed bound. Comparing both laws through the uniform law on direct $t$-subsets proves the conditional distribution statement. The exact conditional uniformity on $\mathfrak D_{r,t}$ is the same $\mathrm{GL}(V_N)$-orbit argument, and the span statement follows because $\mathrm{GL}(V_N)$ acts transitively on $tr$-spaces.
\end{proof}

\begin{corollary}[High-target process law]\label{cor:high-target-process}
There is an absolute constant $a_{\mathrm{HT}}>0$ such that, uniformly over all prime powers $q$, for resonant parameters $2\le k\le r$, $m\ge1$, with $\sigma_r:=\frac{\sqrt c}{r}\longrightarrow\infty$. Define $\mathcal Z_{N,r},\mathcal B_r,\Xi_r,$ and $\Pi_r$ exactly as above. For all sufficiently large parameters,
\begin{align*}
 &d_{\mathrm{TV}}\!\left(\mathcal L(\mathcal Z_{N,r}),\mathcal L(\mathcal B_r)\right)
 +d_{\mathrm{TV}}\!\left(\mathcal L(\Xi_r),\mathcal L(\Pi_r)\right)\\
 &\qquad
 +d_{\mathrm{TV}}\!\left(\mathcal L(X_{N,r}),\Pois(\mu_{N,r})\right)
 \le \exp\!\left(-a_{\mathrm{HT}}\sigma_r\log(2+\sigma_r)\right).
\end{align*}
For fixed $q$, this upgrades \cref{thm:poisson-general-intro} from the scalar count to the entire null configuration.
\end{corollary}

\begin{proof}
Take $T=\lfloor\theta_{\mathrm{HT}}\sqrt c/r\rfloor$ with a fixed sufficiently small $\theta_{\mathrm{HT}}>0$. Then $T\to\infty$, $Tr\le N$ for all sufficiently large parameters, and \cref{prop:factorial-error} gives uniformly for $1\le t\le T+1$
\[
 \left|\E(X_r)_t-\mu_r^t\right|\le q^{-b_{\mathrm{HT}}c}
\]
for an absolute constant $b_{\mathrm{HT}}>0$; the proof of that proposition gives the same bound for the non-direct contribution. Choose the fixed $a_{\mathrm{HT}}>0$ sufficiently small for the resulting Brun and process estimates. Quantitative Brun inversion therefore yields
\[
 d_{\mathrm{TV}}\!\left(\mathcal L(X_r),\Pois(\mu_r)\right)
 \le \exp\!\left(-a_{\mathrm{HT}}\sigma_r\log(2+\sigma_r)\right).
\]
Repeating the proof of \cref{prop:uniform-direct-fixed} gives the same bound for the probability that either $\mathcal Z_{N,r}$ or $\mathcal B_r$ is non-direct. The orbit-reduction and Bernoulli-to-Poisson lemmas then give the two process-level estimates.
\end{proof}

\subsection{Maximum law}

\begin{corollary}[Complete maximum-dimension law]
\label{cor:complete-maximum}
Fix $k\ge3$. Let $\alpha_N$ be the maximum null dimension and $M_N=\max\{1\le r\le N:\Delta_N(r)\ge0\}$. Then the following hold.
\begin{enumerate}[label=\textup{(\roman*)}]
\item Fix $q$ and $m\ge1$. Along every sequence $N\to\infty$ for which $\Delta_N(M_N)>0$, $\Pp(\alpha_N=M_N)\longrightarrow1$.

\item Along every sequence for which $\Delta_N(M_N)=0$, put $r=M_N$. If $r\to\infty$ and the resonant sequence is covered by \cref{thm:fixed-target-intro}, allowing $q=q_r$ and $m=m_r$ to vary, then
\[
\Pp(\alpha_N\notin\{r-1,r\})\longrightarrow0,
\]
and
\[
\Pp(\alpha_N=r)=1-e^{-C_q}+o(1),
\qquad
\Pp(\alpha_N=r-1)=e^{-C_q}+o(1).
\]
In particular, if $q_r\to\infty$, these probabilities converge to $1-e^{-1}$ and $e^{-1}$.
\end{enumerate}
\end{corollary}

\begin{proof}
Part \textup{(i)} is \cref{prop:locking}(i).

For \textup{(ii)}, the identity $\Delta_N(r)=0$ is equivalent to $N-r=\frac mr\binom rk$, so the parameters are resonant. By \cref{eq:rplusone-space-probability} and Markov's inequality, $\Pp(\alpha_N\ge r+1)\longrightarrow0$ uniformly. At dimension $r-1$, the first moment exponent is $c(k-1)+r-1\to\infty$, while $N-2(r-1)=c-r+2\to\infty$. Hence the pair-overlap estimate implies that a null $(r-1)$-space exists with probability tending uniformly to one. Thus $\Pp(\alpha_N\notin\{r-1,r\})\longrightarrow0$. On this event, $\alpha_N=r$ if and only if $X_{N,r}>0$, so \cref{eq:uniform-count-to-Cq} gives
\[
\Pp(\alpha_N=r)=1-e^{-C_q}+o(1),
\qquad
\Pp(\alpha_N=r-1)=e^{-C_q}+o(1).
\]
\end{proof}

\section{The bilinear boundary}\label{sec:bilinear-boundary}

\begin{theorem}[Bilinear fixed-target failures]\label{thm:bilinear-boundary-intro}
Fix $q$.
\begin{enumerate}[label=\textup{(\roman*)}]
 \item Let $r\to\infty$ through odd integers, $c=(r-1)/2$, and $N=(3r-1)/2$. For a uniform alternating bilinear form $B:\Lambda^2\F_q^N\to\F_q$, the number $X_{N,r}$ of totally isotropic $r$-spaces satisfies
 \[
 \E X_{N,r}\to C_q,
 \qquad
 \Pp(X_{N,r}>0)\le C_q q^{-(r^2-1)/8}\to0.
 \]
 Hence $X_{N,r}\to0$ in probability.
 \item Let $N=2r-1$ and choose uniformly
 $B=(B_1,B_2):\Lambda^2\F_q^N\to\F_q^2$. Every pair has a common totally isotropic $r$-space, so $X_{N,r}\ge1$ deterministically, while
 \[
 \frac{\E(X_{N,r})_2}{(\E X_{N,r})^2}\ge\frac{r-1}{C_q}.
 \]
\end{enumerate}
\end{theorem}
We prove \cref{thm:bilinear-boundary-intro} after establishing \cref{lem:two-form-common-isotropic}, without using exact rank distributions for random alternating matrices. Such distributions are well known; see, for example, \cite{FulmanGoldstein2015}. For $m=1$ a union bound on the radical is enough, while the $m=2$ obstruction follows from an elementary deterministic induction.

We use the standard Witt-index formula for an alternating form $B$ on an $N$-dimensional space and let $R_B=\dim\operatorname{rad}B$:
\begin{equation}
 \max\{\dim H:H\text{ totally isotropic for }B\}
 =\frac{N+R_B}{2}.
 \label{eq:alternating-witt-index}
\end{equation}
It follows by passing to the nondegenerate symplectic quotient and lifting a Lagrangian subspace; see, for example, \cite{BGH1987}.

\begin{lemma}[Common isotropic subspace for two alternating forms]
\label{lem:two-form-common-isotropic}
Let $B_1,B_2$ be alternating bilinear forms on a vector space $V$ over an arbitrary field. If $\dim V\ge2r-1$, then there is an $r$-dimensional subspace that is totally isotropic for both forms.
\end{lemma}

\begin{proof}
It is enough to treat $\dim V=2r-1$, by restricting both forms to a $(2r-1)$-dimensional subspace. We induct on $r$. The case $r=1$ is immediate.

Choose a nonzero pair $(\lambda,\mu)$ and put $B_{\lambda,\mu}=\lambda B_1+\mu B_2$. An alternating form has even rank, so on the odd-dimensional space $V$ the form $B_{\lambda,\mu}$ is singular; choose $0\ne v\in\operatorname{rad}B_{\lambda,\mu}$. Set $f_i(x)=B_i(v,x)$ for $i=1,2$. Since $\lambda f_1+\mu f_2=0$, the space $W=\ker f_1\cap\ker f_2$ has codimension at most one in $V$ and contains $v$. The vector $v$ lies in the radical of both restricted forms, so they descend to $\overline W=W/\langle v\rangle$, whose dimension is at least $2r-3$.
Choose a $(2r-3)$-dimensional subspace of $\overline W$ if necessary. By the induction hypothesis, the two descended forms have a common totally isotropic $(r-1)$-space $\overline H$. Its inverse image in $W$ is an $r$-space containing $v$, and both $B_1,B_2$ vanish on it.
\end{proof}

\begin{proof}[Proof of \cref{thm:bilinear-boundary-intro}]
For part \textup{(i)}, at $k=2$, $m=1$, and $c=(r-1)/2$, the resonance identity $\binom r2=rc$ holds. Since $r,c\to\infty$, \cref{eq:mean-resonance} gives $\E X_{N,r}\to C_q$. If $X_{N,r}>0$, then \cref{eq:alternating-witt-index} implies $r\le(N+R_B)/2$, hence $R_B\ge2r-N=(r+1)/2$ because $N=(3r-1)/2$.

Put $\sigma=(r+1)/2$. For a fixed $\sigma$-space $W\le\F_q^N$, the condition $W\le\operatorname{rad}B$ imposes $\binom N2-\binom{N-\sigma}{2}=\sigma N-\sigma(\sigma+1)/2$ independent linear conditions on $B$. Hence, by a union bound and \cref{eq:qbinom-bounds},
\begin{align*}
 \Pp(X_{N,r}>0)
 &\le \Pp(R_B\ge \sigma)\\
 &\le \qbinom N\sigma q^{-\sigma N+\sigma(\sigma+1)/2}\\
 &\le C_q q^{-\sigma(\sigma-1)/2}\\
 &=C_q q^{-(r^2-1)/8}.
\end{align*}
This proves the probability bound and hence
$X_{N,r}\to0$ in probability. Moreover,
\[
 \E[X_{N,r}\mid X_{N,r}>0]
 =\frac{\E X_{N,r}}{\Pp(X_{N,r}>0)}
 \ge (1-o(1))q^{(r^2-1)/8}.
\]

For part \textup{(ii)}, the resonance parameters are $c=r-1$ and $N=2r-1$. By \cref{lem:two-form-common-isotropic}, every pair $(B_1,B_2)$ has a common totally isotropic $r$-space, so $X_{N,r}\ge1$ deterministically. By \cref{eq:mean-resonance}, $\mu_{N,r}\longrightarrow C_q$.

It remains to prove the second-factorial-moment lower bound. Fix an $r$-space $H_1$. Every distinct $r$-space $H_2$ in a $(2r-1)$-dimensional ambient space has $1\le d_2\le r-1$. Let $\pi_{d_2}$ be the probability that a uniformly random $H_2$ has intersection dimension $d_2$ with $H_1$. By \cref{eq:grassmann-intersection-count},
\[
 \pi_{d_2}=
 \frac{
 \qbinom {r}{d_2}\qbinom{r-1}{r-d_2}q^{(r-d_2)^2}
 }{\qbinom{2r-1}{r}}.
\]
Using the lower Gaussian-binomial bound in the numerator and the upper bound
in the denominator,
\[
 \pi_{d_2}
 \ge
 C_q^{-1}
 q^{d_2(r-d_2)+(d_2-1)(r-d_2)+(r-d_2)^2-r(r-1)}
 =
 C_q^{-1}q^{-d_2(d_2-1)}.
\]
For $k=2,m=2$, the weight in \cref{eq:exact-pair-factorial-moment} is $q^{2\binom {d_2}{2}}=q^{d_2(d_2-1)}$. Hence, using the preceding lower bound on $\pi_{d_2}$,
\[
 \frac{\E(X_{N,r})_2}{\mu_{N,r}^2}
 =
 \sum_{d_2=1}^{r-1}\pi_{d_2}q^{2\binom {d_2}{2}}
 \ge\frac{r-1}{C_q},
\]
which proves the stated second-moment bound. Since $\Pp(X_{N,r}=0)=0$ for every $r$, the law cannot converge to $\Pois(C_q)$, which assigns positive mass to zero.
\end{proof}

\begin{corollary}[Bilinear $m\ge3$ moment obstruction]\label{cor:bilinear-high-moment}
Fix $q$ and $m\ge3$. Let $r\to\infty$ along integers for which $c=m(r-1)/2\in\mathbb Z$, let $X_{N,r}$ count common totally isotropic $r$-spaces of a uniformly random $m$-tuple of alternating bilinear forms. Then $\mu_{N,r}\to C_q$ and $\frac{\E(X_{N,r})_2}{\mu_{N,r}^2}\longrightarrow1$. Nevertheless, for every fixed integer $t>(m+2)/2$,
\[
\log_q\E(X_{N,r})_t\ge\left(t-\frac{m+2}{2}\right)r-O_{m,t}(1),
\]
and hence $\E(X_{N,r})_t\to\infty$. If $X_{N,r+1}$ denotes the number of null $(r+1)$-spaces, then $\Pp(X_{N,r+1}>0)\le C_q q^{-(m+2)(r+1)/2}\to0$. Thus vanishingly rare null $(r+1)$-spaces already force divergence of sufficiently high factorial moments.
\end{corollary}

\begin{proof}
At bilinear resonance, $m\binom r2=rc$ and $c=m(r-1)/2$, so \cref{eq:mean-resonance} gives $\mu_{N,r}\to C_q$. Moreover, $N-2r=c-r=((m-2)r-m)/2\to\infty$ for $m\ge3$. Hence \cref{lem:pair-overlap}, applied at equality in \cref{eq:below-threshold}, gives the stated second-moment limit.

For $k=2$, the exponent in \cref{prop:rplusone-space-moment-obstruction} is $A_r=(r+1)(1+c/(r-1))=(m+2)(r+1)/2$. Therefore \cref{eq:rplusone-space-probability} gives $\Pp(X_{N,r+1}>0)\le \E X_{N,r+1}\le C_q q^{-(m+2)(r+1)/2}$. Finally, for every fixed $t>(m+2)/2$, \cref{eq:rplusone-space-moment-lower} gives
\[
\log_q\E(X_{N,r})_t\ge tr-A_r-O(t)=\left(t-\frac{m+2}{2}\right)r-O_{m,t}(1),
\]
which tends to $+\infty$.
\end{proof}

\begin{remark}[The fixed bilinear regime $m\ge3$]
For fixed $m\ge3$, \cref{cor:bilinear-high-moment} shows that high factorial moments diverge, but this does not rule out convergence in distribution to $\Pois(C_q)$: the divergence is carried by events whose probability tends to zero, so uniform integrability fails. Thus the factorial-moment method does not determine the limiting law in this regime. We conjecture below that, for every fixed $q$ and $m\ge3$, the critical count nevertheless converges to $\Pois(C_q)$; see \cref{conj:bilinear-poisson}.
\end{remark}

Thus exact first-moment resonance alone does not force Poisson behavior. Although the low-excess stability theorem remains valid for $k=2$, the entropy separation that makes it decisive for $k\ge3$ is lost at the bilinear scale.

\section{Sharp thresholds and open problems}
\label{sec:conjectures}

The results above leave three concrete thresholds unresolved. The first is the remaining fixed-target bilinear regime. The second asks whether the obstruction from rare null $(r+1)$-spaces gives the exact factorial-moment transition. The third asks whether the first distance-one exterior dependence also determines the sharp total-variation exponent. Recall that $\chi_k=(k-1)/k!$.

\begin{conjecture}[Bilinear Poisson phase transition]
\label{conj:bilinear-poisson}
Fix a prime power $q$ and an integer $m\ge3$. Let $r\to\infty$ along integers for which $c=m(r-1)/2\in\mathbb Z$. For a uniformly random $\Theta_N:\Lambda^2\F_q^N\to\F_q^m$, let $X_{N,r}$ count the common totally isotropic $r$-spaces. Then $X_{N,r}\xrightarrow{d}\Pois(C_q)$.
\end{conjecture}

This would complete the fixed-target bilinear phase diagram: $m=1$ and $m=2$ are non-Poisson by \cref{thm:bilinear-boundary-intro}, whereas every fixed $m\ge3$ would be Poisson. The restriction to fixed $q$ is deliberate. The conjecture concerns the scalar critical count; the stronger independent Bernoulli model is not asserted here.

\begin{conjecture}[Exact factorial-moment transition]
\label{conj:factorial-threshold}
Fix a prime power $q$ and integers $m\ge1$ and $k\ge3$. Let $r\to\infty$ along resonant integers. Then, for every $0<\varepsilon<\chi_k$,
\[
\sup_{1\le t\le\lfloor(\chi_k-\varepsilon)m r^{k-2}\rfloor}
\left|\frac{\E(X_r)_t}{\mu_r^t}-1\right|
\longrightarrow0.
\]
\end{conjecture}

The opposite side of this proposed transition is already detected by \cref{prop:rplusone-space-moment-obstruction}: for every $\varepsilon>0$, factorial orders in its stated range satisfying $t\ge(\chi_k+\varepsilon)m r^{k-2}$ receive an exponentially large non-Poisson contribution from null $(r+1)$-spaces. For fixed $q,m,k$, the condition $t\le L_r/2$ from that proposition is automatic at this polynomial scale. Thus the conjecture predicts the sharp transition
\[
t_{\mathrm{crit}}\sim\chi_k m r^{k-2}=\frac{k-1}{k!}m r^{k-2},
\]
instead of the smaller stability scale $\rho_{m,k}(r)$ reached by the present proof.

\begin{conjecture}[Sharp dependence exponent]
\label{conj:sharp-tv-exponent}
Fix a prime power $q$ and integers $m\ge1$ and $k\ge3$. Let $r\to\infty$ along resonant integers. Then
\[
\lim_{r\to\infty}
\frac{-\log_q d_{\mathrm{TV}}\!\left(\mathcal L(X_{N,r}),\Pois(\mu_{N,r})\right)}
{m r^{k-1}}
=\chi_k.
\]
The same limit, with the same normalization, holds for the random-set and point-process distances in \cref{eq:bernoulli-subset-tv,eq:poisson-process-tv}.
\end{conjecture}

The upper bound on the possible exponent is already proved in \cref{prop:tv-lower-obstruction}. Moreover, \cref{cor:exterior-dependence-gaps} shows that $\chi_k=(k-1)/k!$ is exactly the first possible exterior-dependence cost and is attained by a distance-one pair. The remaining direction is to show that the aggregate contribution of all more complicated dependencies does not occur on a larger probability scale.

\appendix
\section{Sharp quantized stability refinements}\label{app:quantized-stability}
We retain the notation and standing assumptions of \cref{sec:low-excess-stability}.

\begin{proposition}[Quantized endpoint-defect budget]\label{prop:quantized-endpoint-budget}
Let $\mathcal H=(H_1,\ldots,H_t)$ be a $\Gamma$-low-excess family. There exists $r_0=r_0(k,\Gamma)$ such that, whenever $r\ge r_0$ and $r>2D_{k,\Gamma}$, the following holds. Set
\[
  \mathcal N=\{j\ge2:d_j\le D_{k,\Gamma}\},\qquad
  \mathcal A=\{j\ge2:r-d_j\le D_{k,\Gamma}\}.
\]
Then
\begin{equation}
 \sum_{j\in\mathcal N}d_j
 +(k-1)\sum_{j\in\mathcal A}(r-d_j)
 \le \lfloor k!\Gamma\rfloor.\label{eq:quantized-endpoint-budget}
\end{equation}
More precisely,
\begin{equation}
 \frac{G_k(\mathcal H)}{r^{k-1}}
 \ge \frac1{k!}\left(
 \sum_{j\in\mathcal N}d_j
 +(k-1)\sum_{j\in\mathcal A}(r-d_j)
 \right)-O_{k,\Gamma}(r^{-1}).\label{eq:weighted-endpoint-asymptotic}
\end{equation}
Thus one unit of low-end overlap costs asymptotically $1/k!$ at the natural scale, while one unit of high-end defect costs $(k-1)/k!$. In particular, there are at most $\lfloor k!\Gamma\rfloor$ indices with $1\le d_j\le D_{k,\Gamma}$, and at most $\left\lfloor\dfrac{\lfloor k!\Gamma\rfloor}{k-1}\right\rfloor$ indices with $1\le r-d_j\le D_{k,\Gamma}$.
\end{proposition}

\begin{proof}
Recall $f_{r,k}(d)=\beta d-\binom dk$. By \cref{eq:fixed-excess-sum,lem:overlap-profile}, $G_k(\mathcal H)\ge\sum_{j=2}^t f_{r,k}(d_j)$. By \cref{lem:low-excess-endpoints}, $\sum_{j=2}^t\min\{d_j,r-d_j\}=O_{k,\Gamma}(1)$, so only $O_{k,\Gamma}(1)$ indices have positive endpoint defect. Uniformly for fixed $0\le d\le D_{k,\Gamma}$,
\[
 f_{r,k}(d)=\frac{d}{k!}r^{k-1}+O_{k,\Gamma}(r^{k-2}),
\]
while, uniformly for fixed $0\le h\le D_{k,\Gamma}$,
\[
 f_{r,k}(r-h)=\frac{(k-1)h}{k!}r^{k-1}+O_{k,\Gamma}(r^{k-2}).
\]
At $d=0$ and $h=0$ the formulas are exact, so summing only over the positive-defect indices gives \cref{eq:weighted-endpoint-asymptotic}. Combining with $G_k(\mathcal H)\le \Gamma r^{k-1}$ yields
\[
 \sum_{j\in\mathcal N}d_j
 +(k-1)\sum_{j\in\mathcal A}(r-d_j)
 \le k!\Gamma+O_{k,\Gamma}(r^{-1}).
\]
The left side is integral, which gives \cref{eq:quantized-endpoint-budget}.
\end{proof}

\begin{corollary}[Sharp quantized low-excess phase]\label{cor:sharp-pairwise-phase-intro}
Fix $k\ge2$ and $\Gamma>0$, and put $J_\Gamma=\lfloor k!\Gamma\rfloor$ and $B_\Gamma=\left\lfloor\frac{k!\Gamma}{k-1}\right\rfloor$. Every ordering of a $\Gamma$-low-excess family $\mathcal H$ and canonical partition $[t]=P_1\sqcup\cdots\sqcup P_\kappa$ satisfies:
\begin{enumerate}[label=\textup{(\roman*)}]
 \item pairs in one block have Grassmann distance at most $B_\Gamma$, while pairs in different blocks have intersection dimension at most $J_\Gamma$;
 \item every block has an explicit bounded-thickness top, $r\le \dim\sum_{i\in P_a}H_i\le r+B_\Gamma$;
 \item for every choice of representatives $A_a=H_{i_a}$ with $i_a\in P_a$, $0\le \kappa r-\dim(A_1+\cdots+A_\kappa)\le J_\Gamma$;
 \item $\kappa r-J_\Gamma\le \mathfrak O\le\kappa r+B_\Gamma$.
\end{enumerate}
The constants are sharp: a fixed intersection-$d$ pair belongs to the class whenever $\Gamma\ge d/k!$, and a fixed distance-$h$ pair belongs whenever $\Gamma\ge h(k-1)/k!$; below the corresponding value, that pair type is eventually excluded. In particular, the upper block-thickness bound $B_\Gamma$ and the representative-span-deficiency bound $J_\Gamma$ are attained by two-space families whenever the corresponding integer is positive.

For $k\ge3$, this produces two distinct low-excess transitions. If $\Gamma<1/k!$, then $J_\Gamma=B_\Gamma=0$ and every $\Gamma$-low-excess family is an ordinary direct sum. If $1/k!\le\Gamma<(k-1)/k!$, then $B_\Gamma=0$, so every canonical block is a singleton, while $0\le tr-\mathfrak O\le J_\Gamma$. Thus bounded intersections may already occur in this intermediate range, but no nontrivial Grassmann cluster can occur. The first departure from ordinary directness occurs at $\Gamma=1/k!$, whereas genuine distance-one attachments first become admissible at $\Gamma=(k-1)/k!$. For $k=2$ these two thresholds coincide at $1/2$, so the intermediate phase collapses.
\end{corollary}

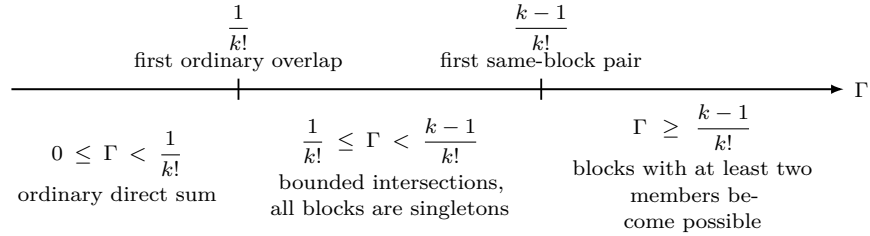
\begin{figure}[htbp]
\centering
\begin{tikzpicture}[
 x=1cm,y=1cm,
 every node/.style={font=\scriptsize},
 tick/.style={line width=0.7pt},
 arr/.style={-{Latex[length=1.7mm]},line width=0.75pt}
]

\draw[arr] (0,0) -- (11,0) node[right] {$\Gamma$};

\draw[tick] (3,0.14) -- (3,-0.14);
\draw[tick] (7,0.14) -- (7,-0.14);

\node[above=3pt,align=center] at (3,0)
 {$\displaystyle \frac1{k!}$\\
 first ordinary overlap};

\node[above=3pt,align=center] at (7,0)
 {$\displaystyle \frac{k-1}{k!}$\\
 first same-block pair};

\node[align=center,text width=2.7cm] at (1.4,-1.05)
 {$\displaystyle 0\le\Gamma<\frac1{k!}$\\[3pt]
 ordinary direct sum};

\node[align=center,text width=3.4cm] at (5,-1.05)
 {$\displaystyle
 \frac1{k!}\le\Gamma<\frac{k-1}{k!}$\\[3pt]
 bounded intersections,\\
 all blocks are singletons};

\node[align=center,text width=3.2cm] at (9,-1.05)
 {$\displaystyle
 \Gamma\ge\frac{k-1}{k!}$\\[3pt]
 blocks with at least two\\
 members become possible};

\end{tikzpicture}

\caption{Sharp low-excess phases for fixed $k\ge3$. Ordinary non-directness first becomes possible at $\Gamma=1/k!$. In the intermediate regime, distinct $r$-spaces may have bounded intersection, but all clusters are singletons. At $\Gamma=(k-1)/k!$, multi-member clusters become possible.}
\label{fig:sharp-low-excess-thresholds}
\end{figure}

\begin{proof}[Proof of \cref{cor:sharp-pairwise-phase-intro}]
By heredity of the exterior excess, \cref{prop:quantized-endpoint-budget} applies to every subfamily and after every reordering.

Take two distinct members $H_1,H_2$. Applying the endpoint dichotomy and \cref{prop:quantized-endpoint-budget} to this pair gives either
\[
\dim(H_1\cap H_2)\le J_\Gamma
\qquad\text{or}\qquad
d_{\mathrm{Gr}}(H_1,H_2)\le
\left\lfloor\frac{J_\Gamma}{k-1}\right\rfloor
=B_\Gamma.
\]
By \cref{cor:low-excess-normal-form-intro}, two members of the same canonical block have distance $<r/2$, so they must satisfy the second alternative. Conversely, members of distinct blocks cannot satisfy the second alternative, since $B_\Gamma<r/2$ and they would then be adjacent in the graph defining the canonical partition. This proves \textup{(i)}.

Fix a block $P_a$ and order its members with any one of them first. By \textup{(i)}, every later member meets the first in dimension at least $r-B_\Gamma$, and hence is an attachment step. Therefore the span-increment formula and \cref{prop:quantized-endpoint-budget} give
\[
\dim\sum_{i\in P_a}H_i-r
=\sum_{j\ge2}(r-d_j)
\le
\left\lfloor\frac{J_\Gamma}{k-1}\right\rfloor
=B_\Gamma,
\]
which proves \textup{(ii)}.

For \textup{(iii)}, choose one representative $A_a$ from each block and apply the budget to this representative subfamily. Every representative after the first is a new-block step: if one were an attachment, \cref{lem:bounded-radius-clustering} would place it within bounded Grassmann distance of an earlier new-block representative, hence in the same canonical block. Thus $0\le\kappa r-\dim(A_1+\cdots+A_\kappa)=\sum_{j=2}^{\kappa}d_j\le J_\Gamma$.

For \textup{(iv)}, order the whole family block by block, placing a representative first in each block. By the preceding argument the $\kappa$ representatives are exactly the new-block steps, while by \textup{(i)} every remaining member is an attachment step. Hence $\mathfrak O-\kappa r=-\sum_{\substack{j\ge2\\ H_j\text{ new}}}d_j+\sum_{\substack{j\ge2\\ H_j\text{ attach}}}(r-d_j)$, where \cref{prop:quantized-endpoint-budget} gives $\sum_{\substack{j\ge2\\ H_j\text{ new}}}d_j+(k-1)\sum_{\substack{j\ge2\\ H_j\text{ attach}}}(r-d_j)\le J_\Gamma$. Since both sums are nonnegative, $-J_\Gamma\le\mathfrak O-\kappa r\le B_\Gamma$, proving \textup{(iv)}.

It remains to prove sharpness. For two $r$-spaces, \cref{eq:exterior-intersection} and the definition of $f_{r,k}$ give
\begin{equation}
G_k((H_1,H_2))=f_{r,k}(d_2).
\label{eq:pair-exact-excess}
\end{equation}
If the intersection dimension $d$ is fixed, then $f_{r,k}(d)\le \beta d\le \frac{d}{k!}r^{k-1}$ and $\frac{f_{r,k}(d)}{r^{k-1}}\longrightarrow\frac{d}{k!}$. Thus such a pair is $\Gamma$-low-excess whenever $\Gamma\ge d/k!$, and is eventually excluded when $\Gamma<d/k!$.

Likewise, for fixed Grassmann distance $h\ge1$,
\[
\begin{aligned}
f_{r,k}(r-h)
&=\sum_{a=0}^{h-1}\binom{r-1-a}{k-1}
-\frac hk\binom{r-1}{k-1} \\
&\le \frac{(k-1)h}{k}\binom{r-1}{k-1}
\le \frac{(k-1)h}{k!}r^{k-1},
\end{aligned}
\qquad
\frac{f_{r,k}(r-h)}{r^{k-1}}
\longrightarrow\frac{(k-1)h}{k!}.
\]
Hence the distance-$h$ threshold is exactly $(k-1)h/k!$. Taking $d=J_\Gamma$ when $J_\Gamma>0$ gives a two-space family with representative-span deficiency exactly $J_\Gamma$, while taking $h=B_\Gamma$ when $B_\Gamma>0$ gives a single canonical block whose span has dimension exactly $r+B_\Gamma$. Thus both constants are attained.

Finally, if $k\ge3$ and $\Gamma<1/k!$, then $J_\Gamma=B_\Gamma=0$. By \textup{(i)} every block is a singleton, so $\kappa=t$, and \textup{(iv)} gives $\mathfrak O=tr$. If $1/k!\le\Gamma<(k-1)/k!$, then $B_\Gamma=0$, so again $\kappa=t$, while \textup{(iv)} gives $0\le tr-\mathfrak O\le J_\Gamma$. For $k=2$, the two thresholds coincide at $1/2$.
\end{proof}

\begin{corollary}[Exterior-dependence gaps]
\label{cor:exterior-dependence-gaps}
Fix $k\ge3$ and let $\mathcal H=(H_1,\ldots,H_t)$ be a family of distinct $r$-spaces. Suppose that the exterior subspaces are not in direct-sum position, i.e. $\mathfrak E<t\binom rk$. Then, for every $\varepsilon>0$ and all sufficiently large $r$,
\begin{equation}
G_k(\mathcal H)\ge\left(\frac{k-1}{k!}-\varepsilon\right)r^{k-1}.
\label{eq:first-exterior-dependence-gap}
\end{equation}
The constant $(k-1)/k!$ is sharp.

If, in addition,
\begin{equation}
\dim(H_i\cap H_j)<k\qquad(i\ne j),
\label{eq:pairwise-exterior-disjoint}
\end{equation}
equivalently $\Lambda^kH_i\cap\Lambda^kH_j=0$ for $i\ne j$, then the stronger bound
\begin{equation}
G_k(\mathcal H)\ge\left(\frac1{(k-1)!}-\varepsilon\right)r^{k-1}
\label{eq:second-exterior-dependence-gap}
\end{equation}
holds for all sufficiently large $r$.
\end{corollary}

\begin{proof}
Since $G_k(\mathcal H)\ge0$, the claim is trivial for $\varepsilon\ge(k-1)/k!$. Thus assume $0<\varepsilon<(k-1)/k!$ and set $\Gamma=(k-1)/k!-\varepsilon$. Suppose, contrary to the claim, that $G_k(\mathcal H)\le\Gamma r^{k-1}$. By \cref{cor:sharp-pairwise-phase-intro}, $B_\Gamma=\lfloor k!\Gamma/(k-1)\rfloor=0$, so every canonical block is a singleton and hence $\kappa=t$. Also $J_\Gamma=\lfloor k!\Gamma\rfloor\le k-2$, and therefore $0\le tr-\mathfrak O\le k-2<k$. Hence \cref{thm:exterior-intro} gives $\mathfrak E\ge t\binom rk-\binom{tr-\mathfrak O}{k}=t\binom rk$, contrary to exterior dependence. This proves \cref{eq:first-exterior-dependence-gap}.

For sharpness, take two $r$-spaces with Grassmann distance one, so their intersection has dimension $r-1$. By \cref{eq:pair-exact-excess}, $G_k(H_1,H_2)=f_{r,k}(r-1)$ and $\frac{f_{r,k}(r-1)}{r^{k-1}}\longrightarrow\frac{k-1}{k!}$. Moreover, $\Lambda^kH_1\cap\Lambda^kH_2=\Lambda^k(H_1\cap H_2)\ne0$, so the exterior sum is dependent. This proves sharpness.

Now assume \cref{eq:pairwise-exterior-disjoint}. Since $G_k(\mathcal H)\ge0$, the claim is trivial for $\varepsilon\ge1/(k-1)!$. Thus assume $0<\varepsilon<1/(k-1)!$ and set $\Gamma=1/(k-1)!-\varepsilon$. Suppose, contrary to the claim, that $G_k(\mathcal H)\le\Gamma r^{k-1}$. Then $J_\Gamma\le k-1$ and $B_\Gamma=\left\lfloor\frac{k!\Gamma}{k-1}\right\rfloor\le1$. If a canonical block contained two distinct members, \cref{cor:sharp-pairwise-phase-intro}(i) would give Grassmann distance at most one. Distinctness would therefore force distance exactly one, and hence intersection dimension $r-1\ge k$ for large $r$, contradicting \cref{eq:pairwise-exterior-disjoint}. Thus every canonical block is again a singleton and $\kappa=t$. Therefore $0\le tr-\mathfrak O\le J_\Gamma\le k-1<k$. Applying \cref{thm:exterior-intro} exactly as above contradicts $\mathfrak E<t\binom rk$. This proves \cref{eq:second-exterior-dependence-gap}.
\end{proof}


\begin{thebibliography}{99}

\bibitem{BarbourBrown1992}
A.~D.~Barbour and T.~C.~Brown,
\newblock Stein's method and point process approximation,
\newblock \emph{Stochastic Process. Appl.} \textbf{43} (1992), no.~1, 9--31,
\newblock doi:10.1016/0304-4149(92)90073-Y.

\bibitem{FeldmanPropp1992}
D.~Feldman and J.~Propp,
\newblock A linear Ramsey theorem,
\newblock \emph{Adv. Math.} \textbf{95} (1992), no.~1, 1--7,
\newblock doi:10.1016/0001-8708(92)90041-I.

\bibitem{ChenXuYe2026}
Q.~Chen, Z.~Xu, and K.~Ye,
\newblock Tur\'an problems for multilinear maps,
\newblock arXiv:2603.00715, 2026.

\bibitem{ChenYe2025}
Q.~Chen and K.~Ye,
\newblock Isotropy and completeness indices of multilinear maps,
\newblock arXiv:2510.27387, 2025.

\bibitem{ChenYeGeometry2026}
Q.~Chen and K.~Ye,
\newblock Geometry of multilinear varieties over infinite fields and its applications,
\newblock arXiv:2605.04859, 2026.

\bibitem{Tevelev2001}
E.~A.~Tevelev,
\newblock Isotropic subspaces of polylinear forms,
\newblock \emph{Math. Notes} \textbf{69} (2001), nos.~5--6, 845--852,
\newblock doi:10.1023/A:1010294818389.

\bibitem{Anzaldo2026}
L.~B.~Anzaldo,
\newblock Isotropic subspaces of Schur modules,
\newblock \emph{Comm. Algebra} \textbf{54} (2026), no.~11, 3544--3553,
\newblock doi:10.1080/00927872.2025.2602887.

\bibitem{EberhardSabatini2025}
S.~Eberhard and L.~Sabatini,
\newblock Probabilistic construction of some extremal $p$-groups,
\newblock \emph{J. Algebra} \textbf{682} (2025), 463--480,
\newblock doi:10.1016/j.jalgebra.2025.06.013.

\bibitem{Qiao2023}
Y.~Qiao,
\newblock Tur\'an and Ramsey problems for alternating multilinear maps,
\newblock \emph{Discrete Anal.} (2023), Paper No.~12, 22 pp.,
\newblock doi:10.19086/da.84736.

\bibitem{ConlonPohoataZakharov2021}
D.~Conlon, C.~Pohoata, and D.~Zakharov,
\newblock Random multilinear maps and the Erd\H{o}s box problem,
\newblock \emph{Discrete Anal.} (2021), Paper No.~17, 8 pp.,
\newblock doi:10.19086/da.28336.

\bibitem{FulmanGoldstein2015}
J.~Fulman and L.~Goldstein,
\newblock Stein's method and the rank distribution of random matrices over finite fields,
\newblock \emph{Ann. Probab.} \textbf{43} (2015), no.~3, 1274--1314,
\newblock doi:10.1214/13-AOP889.

\bibitem{BGH1987}
J.~Buhler, R.~Gupta, and J.~Harris,
\newblock Isotropic subspaces for skewforms and maximal abelian subgroups of $p$-groups,
\newblock \emph{J. Algebra} \textbf{108} (1987), no.~1, 269--279,
\newblock doi:10.1016/0021-8693(87)90138-4.

\bibitem{GhorpadePatilPillai2009}
S.~R.~Ghorpade, A.~R.~Patil, and H.~K.~Pillai,
\newblock Decomposable subspaces, linear sections of Grassmann varieties, and higher weights of Grassmann codes,
\newblock \emph{Finite Fields Appl.} \textbf{15} (2009), no.~1, 54--68,
\newblock doi:10.1016/j.ffa.2008.08.001.

\bibitem{Kinser2011}
R.~Kinser,
\newblock New inequalities for subspace arrangements,
\newblock \emph{J. Combin. Theory Ser. A} \textbf{118} (2011), no.~1, 152--161,
\newblock doi:10.1016/j.jcta.2009.10.014.

\bibitem{ScottWilmer2021}
A.~Scott and E.~Wilmer,
\newblock Combinatorics in the exterior algebra and the Bollob\'as Two Families Theorem,
\newblock \emph{J. Lond. Math. Soc.} \textbf{104} (2021), no.~4, 1812--1839,
\newblock doi:10.1112/jlms.12484.

\end{thebibliography}
\end{document}